\documentclass[a4paper]{amsart}
\usepackage{indentfirst}
\usepackage{amsmath}
\usepackage{amsthm}
\usepackage{amssymb}
\usepackage{tikz}
\usepackage{adjustbox}
\usepackage{hyperref}
\usepackage{cleveref}
\usepackage{tikz-cd}

\usepackage{stackengine} 

\usepackage[T1]{fontenc}% http://ctan.org/pkg/fontenc
\usepackage[outline]{contour}% http://ctan.org/pkg/contour
\usepackage{xcolor}% http://ctan.org/pkg/xcolor

\contourlength{0.3pt}
\contournumber{15}

\usepackage{scalerel}
\usepackage{dsfont}
\usepackage{txfonts} %double inclined arrows
\usepackage{bbm} %lowercase bb
\usepackage{bbold} %numbers bb
\usepackage{mathrsfs} %scr
\usepackage[scr=boondox]{mathalfa} %lowercase scr
\DeclareFontFamily{U}{mathc}{}
\DeclareFontShape{U}{mathc}{m}{it}%
{<->s*[1.03] mathc10}{}
\DeclareMathAlphabet{\mathcal}{U}{mathc}{m}{it}

\DeclareMathAlphabet{\matheus}{U}{eus}{m}{n}
\SetMathAlphabet{\matheus}{bold}{U}{eus}{b}{n}

\newtheorem{defi}{Definition}[subsection]

\newtheorem{theorem}[defi]{Theorem}

\AfterEndEnvironment{maintheorem}{\noindent\ignorespaces}

\newtheorem{pro}[defi]{Proposition}
\newtheorem{lemma}[defi]{Lemma}
\theoremstyle{definition}
\newtheorem{construct}[defi]{Construction}
\newtheorem{example}[defi]{Example}
\newtheorem{innercustomgeneric}{\customgenericname}
\providecommand{\customgenericname}{}
\newcommand{\newcustomtheorem}[2]{%
	\newenvironment{#1}[1]
	{%
		\renewcommand\customgenericname{#2}%
		\renewcommand\theinnercustomgeneric{##1}%
		\innercustomgeneric
	}
	{\endinnercustomgeneric}
}
\newcustomtheorem{customdef}{Definition}
\newcustomtheorem{customthm}{Theorem}
\newcustomtheorem{custompro}{Proposition}
\newcustomtheorem{customlem}{Lemma}
\newcustomtheorem{custommainthm}{Main Theorem}

\theoremstyle{remark}
\newtheorem{rk}[defi]{Remark}
\newtheorem{nota}[defi]{Notation}
\newcustomtheorem{customrk}{Remark}

\newcommand{\bbA}{\mathbbm{A}}
\newcommand{\scrA}{\mathscr{A}}
\newcommand{\frA}{\mathfrak{A}}
\newcommand{\twoA}{\mathcal{A}}
\newcommand{\bbB}{\mathbbm{B}}
\newcommand{\scrB}{\mathscr{B}}

\newcommand{\twoB}{\mathcal{B}}

\newcommand{\bbK}{\mathbbm{K}}

\newcommand{\twoC}{\mathcal{C}}
\newcommand{\bbD}{\mathbbm{D}}
\newcommand{\scrD}{\mathscr{D}}

\newcommand{\twoD}{\mathcal{D}}
\newcommand{\F}{\mathscr{F}}
\newcommand{\bbF}{\mathbbm{F}}

\newcommand{\frX}{\mathfrak{X}}

\newcommand{\frs}{\mathfrak{s}}
\newcommand{\frp}{\mathfrak{p}}
\newcommand{\frq}{\mathfrak{q}}
\newcommand{\frm}{\mathfrak{m}}
\newcommand{\frt}{\mathfrak{t}}
\newcommand{\fri}{\mathfrak{i}}
\newcommand{\frz}{\mathfrak{z}}
\newcommand{\frj}{\mathfrak{j}}
\newcommand{\enrC}{\matheus{C}}
\newcommand{\enrD}{\matheus{D}}
\newcommand*{\monV}{\text{\boldmath{{$\matheus{V}$}}}}

\newcommand*{\Cat}{\mathbf{Cat}}

\newcommand*{\twoCat}{\mathbf{2}$-$\mathbf{Cat}}

\newcommand*{\FCat}{\mathbf{\F}$-$\mathbf{Cat}}

\newcommand*{\SigmaCocone}{\mathbf{MarkedCocone}}

\DeclareMathOperator{\im}{im\;}

\DeclareMathOperator{\inc}{inc}

\newcommand*{\TAlg}{\mathrm{T}$-$\mathrm{Alg}}
\newcommand*{\FTAlg}{\mathrm{T}$-$\mathbbm{Alg}}

\newcommand{\sigmalim}[2]{\mathrm{m}_{#1}\mathrm{lim\;}{#2}}

\newcommand{\dotlim}[2]{\mathrm{d}_{#1}\mathrm{lim\;}{#2}}

\newcommand{\twoEl}[1]{\mathcal{El}{#1}}
\newcommand{\scrEl}[1]{\mathscr{El}{#1}}

\newcommand{\bbEl}[1]{\mathbbm{El}{#1}}
\newcommand{\Lan}[2][]{\mathrm{Lan}_{#2}{#1}}

\newcommand*{\shortbar}{\scalebox{1.6}[1]{-}}

\definecolor{turquoise}{HTML}{00B4CE}

\usepackage{lmodern}
\usepackage{mathtools,enumitem}
\usepackage{microtype}
\usepackage{graphicx}

\usetikzlibrary{arrows,calc,decorations.pathmorphing}
\tikzset{
	marked/.style = {decoration = {markings, mark = at position 0.5 with { 
				\node[transform shape, xscale = .8, yscale=.4] {/}; } }, 
				postaction = 
		{decorate} },
	circled/.style = {decoration = {markings, mark = at position 0.5 with { 
				\node[transform shape, scale = .7] {$\circ$}; } }, postaction = 
				{decorate} 
	},
	loose/.style ={->,
		decorate,
		decoration={snake,amplitude=.4mm,segment length=2mm,post length=1mm}},
}

\tikzcdset{scale cd/.style={every label/.append style={scale=#1},
		cells={nodes={scale=#1}}}}
\tikzcdset{every label/.append style = {font = \footnotesize}} % bigger labels

\usetikzlibrary{decorations.markings,decorations.pathmorphing,shapes}
\tikzset{double line with arrow/.style 
	args={#1,#2}{decorate,decoration={markings,%
			mark=at position 0 with {\coordinate (ta-base-1) at (0,1pt);
				\coordinate (ta-base-2) at (0,-1pt);},
			mark=at position 1 with {\draw[#1] (ta-base-1) -- (0,1pt);
				\draw[#2] (ta-base-2) -- (0,-1pt);
}}}}

\tikzset{dar/.style={double,double equal sign distance,-implies}}%style for 
\tikzset{mid/.style={anchor=mid}} % put labels on the arrow

\makeatletter
\newbox\dottedarrow@box
\setbox\dottedarrow@box\hbox
{%
	\begin{tikzpicture}
		\draw[shorten <=0.2cm, shorten >= 0.15cm , dotted,->] (-0.1,0) -- 
		(1.8em,0);
	\end{tikzpicture}%
}
\newcommand*\dottedarrow
{\relax\ifmmode\expandafter\dottedarrow@m\else\expandafter\dottedarrow@t\fi}
\newcommand*\dottedarrow@t[1][1.3em]
{\resizebox{#1}{!}{\raisebox{.5ex}{\usebox\dottedarrow@box}}}
\newcommand*\dottedarrow@m[1][]
{%
	\if\relax\detokenize{#1}\relax
	\mathchoice% values are trial and error based\ldots
	{\dottedarrow@t}
	{\dottedarrow@t}
	{\dottedarrow@t[0.5em]}
	{\dottedarrow@t[0.5em]}%
	\else
	\dottedarrow@t[#1]%
	\fi
}

\newcommand*\amsmyNeg[2][0mu]{\Neginternal{#1}{\amsnegslash}{#2}}
\newcommand*\Neginternal[3]{\mathpalette\Neg@{{#1}{#2}{#3}}}
\newcommand*\Neg@[2]{\Neg@@{#1}#2}
\newcommand*\Neg@@[4]{%
	\mathrel{\ooalign{%
			$\m@th#1#4$\cr
			\hidewidth$\m@th#3{#1}\mkern\muexpr#2*2$\hidewidth\cr
	}}%
}
\newcommand*\mynegslash[1]{\rotatebox[origin=c]{60}{$\m@th#1\shortbar$}}
\newcommand*\amsnegslash[1]{\rotatebox[origin=c]{60}{$\m@th#1-$}}

\newcommand*{\amsnleadsto}{\amsmyNeg[-2mu]{\leadsto}}
\makeatother

\newcounter{sarrow}
\newcommand\xleadsto[1]{%
	\stepcounter{sarrow}%
	\mathrel{\begin{tikzpicture}[decoration=snake]
			\node (\thesarrow) {$\scriptstyle #1$};
			\draw[->,decorate] (\thesarrow.south west) -- (\thesarrow.south 
			east);
		\end{tikzpicture}
	}
}

\renewcommand{\leadsto}{\xleadsto{\mkern9mu}}

\usetikzlibrary{graphs,graphs.standard,quotes}

\setlist[enumerate,1]{label=\textup{(\arabic*)}}
\setlist[enumerate,2]{label=\textup{(\alph*)}}

\usepackage{thmtools}

\newcommand{\longref}[2]{\hyperref[#2]{#1~\textup{\ref*{#2}}}}
\newcommand{\eqtref}[2]{\hyperref[#2]{#1~\textup{(\ref*{#2})}}}
\newcommand{\longcite}[1]{\textup{\cite{#1}}}

\newcommand*{\op}{\mathrm{op}}%opposite
\DeclareMathOperator{\ob}{ob}
\DeclareMathOperator{\mor}{mor}

\numberwithin{equation}{section}
\newcounter{subsubsubsection}[subsubsection]

\makeatletter
\renewcommand\paragraph{\@startsection{paragraph}{5}{\z@}%
	{3.25ex \@plus1ex \@minus.2ex}%
	{-1em}%
	{\normalfont\normalsize\bfseries}}
\renewcommand\subparagraph{\@startsection{subparagraph}{6}{\parindent}%
	{3.25ex \@plus1ex \@minus .2ex}%
	{-1em}%
	{\normalfont\normalsize\bfseries}}
\def\toclevel@subsubsubsection{4}
\def\toclevel@paragraph{5}
\def\toclevel@paragraph{6}
\def\l@subsubsubsection{\@dottedtocline{4}{7em}{4em}}
\def\l@paragraph{\@dottedtocline{5}{10em}{5em}}
\def\l@subparagraph{\@dottedtocline{6}{14em}{6em}}
\makeatother

\usepackage[margin=22.4mm]{geometry} 
\title{Dotted $2$-limits}
\author{Joanna Ko}
\email{joanna.ko.maths@gmail.com}
\address{Department of Mathematics and Statistics, Masarykova univerzita, 
	Kotlářská 2, Brno 61137, Czech Republic}
\keywords{enhanced $2$-limit; marked limit; codescent objects; 2-monads; enhanced 
	$2$-category theory; $\F$-categories}

\begin{document}
	
	\begin{abstract}
		Marked limits, or Cartesian quasi-limits introduced by Gray, give an 
		alternative approach to $\Cat$-weighted limits in $2$-category theory.  
		This was first established by Street, and we aim to give a new approach 
		to this 
		result using marked codescent objects of marked coherence data which we 
		introduce in this article.  We then propose the notion of dotted 
		$2$-limits, which is a natural generalisation of marked limits to the 
		enhanced $2$-categorical setting. We establish that dotted $2$-limits and 
		$\F$-weighted limits both have the same expressive power.
	\end{abstract}
	\maketitle
	\tableofcontents
	
	\section{Introduction}
	\label{sec:intro}
	Marked limits\footnote{We prefer the name marked limits, because it 
		hints that certain collection of morphisms are special and recorded. We 
		borrow this terminology from marked simplicial sets. Apart from that, 
		we 
		discovered that in \longcite{GHL:2022}, the authors used the name 
		\emph{marked limits} for the case which replaces the 
		identity 
		$2$-cells in our case with invertible $2$-cells. This notion was first 
		considered in \longcite{book:Gray:1974}, and recently in  
		\longcite{DDS:2018}. We believe that this 
		terminology is convenient and accurate in depicting the situation. } of 
	$2$-functors were 
	first introduced by Gray in \longcite{book:Gray:1974}, under the name 
	\emph{Cartesian quasi-limits}. Later in \longcite{Szyld:2019}, Szyld 
	investigated the notion under the name \emph{$\sigma$-$s$-limits}. In 
	Mesiti's recent work \longcite{Mesiti:2023}, this notion was studied in a 
	more restrictive sense under the name \emph{lax 
		normal conical 2-limits}.
	
	Marked limits are gaining attention in $2$-category theory, mainly due to 
	their 
	convenience and simplicity. In contrast, the notion of $\Cat$-weighted 
	limits, though 
	captures all the $2$-categorical limits that we are interested in, is 
	sometimes convoluted, as it is designed for categories enriched in any 
	arbitrary 
	closed symmetric monoidal category $\monV$, but not exclusively for $\Cat$. 
	On 
	the other hand, marked limits often provide a clearer and a more spread-out 
	presentation of $2$-dimensional limits, which makes them more convenient to 
	track and handle. Indeed, marked limits are very similar to lax limits, in 
	the 
	sense that both of them consider non-strict natural transformations; while 
	the latter considers lax natural transformations alone, the former 
	considers lax natural transformations that restricts to $2$-natural 
	transformations with respect to a specific class of morphisms in the domain 
	$2$-category. This 
	makes the notion more flexible, by allowing different kind of shapes for 
	the cones rather 
	than only triangles, and is the reason why marked limits should be able to 
	capture all $2$-categorical limits of interest.
	
	Seeing that marked limits possess some advantages over $\Cat$-weighted 
	limits, it is then crucial to show that marked limits are as expressive as 
	$\Cat$-weighted limits, so that any $2$-categorical limits of interest can 
	be described using marked limits.
	
	It is well-known that a $\Cat$-weighted limit can be expressed as a 
	marked-lax 
	limit with 
	the same universal property. This has first been discussed in Street's 
	paper \longcite{Street:1976}, and also in the recent 
	work \longcite{Mesiti:2023}. The converse, which says that an marked-lax 
	limit can be turned into a 
	$\Cat$-weighted limit with the same universal property, has also been 
	established in \longcite{Street:1976}. In this article, we provide an 
	alternative proof via a new perspective.

	Our proof is done by constructing the weak morphism classifier, which will 
	then 
	give us the left adjoint $()^\ddagger$ to the inclusion $[\twoC, \Cat] 
	\hookrightarrow 
	[\twoC, \Cat]_{l, \Sigma}$ of the category of $2$-functors $\twoC \to \Cat$ 
	and $2$-natural transformations to that of marked-lax natural 
	transformations. The construction of the weak morphism classifier relies on 
	the notion of \emph{marked codescent objects of 
		marked coherence data}, which is a modified version of lax codescent 
	objects 
	of 
	strict coherence data studied by Lack in \longcite{Lack:2002}.
	
	Next, we introduce the notion of \emph{dotted $2$-limits}, which is the 
	enhanced $2$-categorical counterpart of marked limits.
	
	The study of enhanced $2$-category theory began in \longcite{LS:2012} by 
	Lack and Shulman. The main reason for developing the theory is that, 
	indeed, many interesting categorical structures in 
	Categorical Algebra are $T$-algebras, for a $2$-monad $T$ on a $2$-category 
	$\twoA$; most 
	of the time, however, the morphisms between them are the 
	pseudo or lax $T$-morphisms, instead of the strict $T$-morphisms. For 
	instance, lax monoidal functors and pseudo (strong) monoidal functors 
	between monoidal categories are more common than strict monoidal functors. 
	Therefore, studying $\TAlg_w$ for $w = l, p, c$ 
	becomes salient. In 
	particular, whether or not $\TAlg_w$ admits all the limits that $\twoA$ 
	possesses is a key question. In \longcite{Lack:2005}, Lack has shown that 
	the 
	answer to the question in the case of $\TAlg_l$ depends heavily on the 
	relationship between $\TAlg_s$ and $\TAlg_l$. Thus, instead of viewing them 
	as separate $2$-categories, we should combine $\TAlg_s$ and $\TAlg_l$ into 
	one single structure, which is the starting point of $\F$-categories and 
	enhanced $2$-category theory.
	
	In our paper, we propose the notion of dotted $2$-limits, which, very much in 
	the same spirit of marked limits, often provides a more convenient 
	expression and simpler description than $\F$-weighted limits. To illustrate 
	its convenience, we 
	describe several examples of $\F$-categorical limits as dotted 
	limits.
	
	Our main results in the paper:
	\begin{custommainthm}{5.3.6}
		The dotted-lax limit of an $\F$-functor $S \colon \bbD \to 
		\bbA$ has the same universal property as the $\F$-weighted limit 
		$\{\triangle(\mathbf{1})^\#, 
		S\}$.
	\end{custommainthm}
	\begin{custommainthm}{5.3.9}
		Let $\Phi \colon \bbD \to \bbF$ be an $\F$-weight and $S \colon \bbD 
		\to 
		\bbA$ be an $\F$-functor. Let $\Sigma := \{(d \in \mor \bbD, 1)\}$ be 
		the class of 
		morphisms, and $T := \{(D, \delta \in (\Phi D)_\tau)\}$ be a collection 
		of objects in the $\F$-category of elements $\bbEl{\Phi}$ of $\Phi$. 
		The 
		$\F$-weighted limit $\{\Phi, S\}$ has the same universal property as 
		the 
		dotted-lax limit of the $\F$-functor $SP \colon (\bbEl{\Phi}, \Sigma, 
		T) \to 
		\bbA$.
	\end{custommainthm}
	\noindent say that any dotted $2$-limits can be 
	equivalently expressed as $\F$-weighted limits, and vice versa. The 
	latter can be done in a similar fashion as in the $2$-categorical case, 
	whereas the former is done by constructing the tight part for the left 
	adjoint, using mainly left Kan extensions and embedded image factorisations.
	
	The outline of the paper goes as follows. In \longref{Section}{sec:prelim}, 
	we briefly recall the preliminary notions, particularly the basics of 
	$2$-monad theory. In 
	\longref{Section}{sec:marked}, we begin by recalling marked limits and move 
	on to introduce the notion of marked codescent objects of marked coherence 
	data, and then give a complete and detailed proof of that any marked limit 
	has the same universal property 
	as a $\Cat$-weighted limit. In \longref{Section}{sec:F}, we recall the 
	basics of enhanced $2$-category theory, including notions such as 
	$\F$-categories and $\F$-weights. In \longref{Section}{sec:dotted}, we 
	introduce 
	the notion of dotted $2$-limits, and present several important examples; last 
	but not 
	least, we show the equivalence of dotted $2$-limits and $\F$-weighted limits.
	
	\section*{Acknowledgement}
	The author would like to thank {John Bourke} for his clear guidance, 
	beneficient suggestions and inspiring opinions, which helps her get more 
	familiar with enhanced $2$-category theory and advances her knowledge of 
	category theory to a huge extent. The author is also grateful to Giacomo 
	Tendas for his helpful comments and revisions.

	\section{Preliminaries on basic category theory, enriched category theory 
		and 
		\texorpdfstring{$2$}{2}-category theory}
	\label{sec:prelim}
	
	\subsection{Weighted limits}
	A standard reference for enriched category theory is Kelly's book 
	\longcite{book:Kelly:1982}.
	
	Let $\monV$ be a closed symmetric monoidal category, $\enrC, \enrD$ be two 
	$\monV$-categories, $W \colon \enrC \to \monV$ 
	be a $\monV$-weight, and $F \colon \enrC \to \enrD$ be a $\monV$-functor.
	
	\begin{defi}
		\label{def:weighted_lim}
		The \emph{$\monV$-weighted limit $\{W, F\}$ of $F$ by $W$} is 
		characterised 
		by an isomorphism 
		\begin{equation}
			\label{eqt:weighted_lim}
			\enrD(D, \{W, F\}) \cong [\enrC, \monV](W, \enrD(D, F-))
		\end{equation}
		in $\monV$, which is natural in $D \in \ob \enrD$.
	\end{defi}
	
	\begin{rk}
		Dually, let $G \colon \enrC^\op \to \enrD$ be a $\monV$-functor. The 
		\emph{$\monV$-weighted colimit $W * G$ of $G$ by $W$} is characterised 
		by 
		\begin{equation}
			\label{eqt:weighted_colim}
			\enrD(W * G, D) \cong [\enrC, \monV](W, \enrD(G-, D))
		\end{equation}
		in $\monV$, which is natural in $D \in \ob \enrD$.
	\end{rk}
	
	\subsection{\texorpdfstring{$2$}{2}-monad theory}
	\label{subsec:monad}
	
	We recall some notions and results from $2$-dimensional 
	monad theory. For an introduction, the 
	classical work \longcite{BKP:1989} by Blackwell, Kelly, and Power has  
	provided some basics for the subject, apart from this, we suggest the 
	doctoral thesis \longcite{Bourke:2010} by Bourke as a gentle introductory 
	literature on the topic.
	
	\begin{defi}
		\label{def:T-alg}
		Let $(T, \mu, \eta)$ be a $2$-monad. A \emph{strict $T$-algebra} 
		consists 
		of a pair $(A, a)$, where $A$ is an object of 
		$\twoA$ and $a \colon TA \to A$ is a morphism called the 
		\emph{structure 
			morphism}, satisfying the multiplication and unity
		conditions.
	\end{defi}
	
	\begin{defi}
		\label{def:T-mor}
		Let $(A, a)$ and $(B, b)$ be strict
		$T$-algebras. A \emph{lax $T$-morphism} $(f, \overline{f}) \colon (A, 
		a) \to (B, b)$ consists of a morphism 
		$f \colon A \to B$ equipped with a $2$-cell $\overline{f} \colon b 
		\cdot Tf \to fa$, satisfying the multiplicative and the unital 
		coherence conditions, respectively:
		\begin{equation*}
			\begin{tikzcd}[row sep = small, column sep = tiny]
				& T^2B \ar[rr, "Tb"] \ar[dr, shorten <= -0.2cm, 
				shorten >= -0.2cm, "\mu_B"'] & \phantom{1} & TB \ar[dr, shorten 
				<= -0.2cm, 
				shorten >= -0.2cm, "b"] & 
				\\
				T^2A \ar[ur, shorten <= -0.2cm, 
				shorten >= -0.2cm, "T^2f"] \ar[dr, shorten <= -0.2cm, 
				shorten >= -0.2cm, "\mu_{A}"'] & & TB \ar[rr, "b"] & 
				\phantom{1} \ar[dl, phantom, "\Downarrow\overline{f}"] & B
				\\
				& TA \ar[ur, shorten <= -0.2cm, 
				shorten >= -0.2cm, "Tf"] \ar[rr, "a"'] & \phantom{1} & A 
				\ar[ur, shorten <= -0.2cm, 
				shorten >= -0.2cm,
				"f"']&
			\end{tikzcd}
			=
			\begin{tikzcd}[row sep = small, column sep = tiny]
				& T^2B \ar[rr, "Tb"] & \phantom{1} \ar[dl, phantom, "\Downarrow 
				T\overline{f}"] & TB \ar[dr, shorten <= -0.2cm, 
				shorten >= -0.2cm, "b"] \ar[dd, phantom, 
				"\Downarrow\overline{f}"] & 
				\\
				T^2A \ar[ur, shorten <= -0.2cm, shorten >= -0.2cm, "T^2f"] 
				\ar[rr, "Ta"'] \ar[dr, shorten <= -0.2cm, shorten >= 
				-0.2cm, 
				"\mu_{TA}"']& \phantom{1} 
				& TA 
				\ar[ur, shorten <= -0.2cm, shorten >= -0.2cm, "Tf"] \ar[dr, 
				shorten <= -0.2cm, shorten >= -0.2cm, "a"'] & & B
				\\
				& TA \ar[rr, "a"'] & \phantom{1} & A \ar[ur, shorten <= -0.2cm, 
				shorten >= -0.2cm, "f"']&
			\end{tikzcd},
		\end{equation*}
		\begin{equation*}
			\begin{tikzcd}[row sep = tiny, column sep = tiny]
				& B \ar[dr,  shorten <= -0.2cm, 
				shorten >= -0.2cm,  "\eta_B"'] \ar[drrr, shorten >= -0.2cm, 
				bend left, "1"]&  & 
				\phantom{1}  &
				\\
				A \ar[ur,  shorten <= -0.2cm, 
				shorten >= -0.2cm, "f"] \ar[dr,  shorten <= -0.2cm, 
				shorten >= -0.2cm, "\eta_{A}"'] &  & TB \ar[rr, "b"]& 
				\phantom{1}\ar[dl, phantom, "\Downarrow\overline{f}"] & B
				\\
				&  TA \ar[ur, shorten <= -0.2cm, 
				shorten >= -0.2cm,  "Tf"] \ar[rr, "a"'] & \phantom{1} & A 
				\ar[ur, shorten <= -0.2cm, 
				shorten >= -0.2cm, "f"'] &
			\end{tikzcd}
			=
			\begin{tikzcd}[row sep = tiny, column sep = tiny]
				& B  \ar[drrr, shorten >= -0.2cm, bend left, "1"]&  & 
				\phantom{1}  &
				\\
				A \ar[ur,  shorten <= -0.2cm, 
				shorten >= -0.2cm, "f"] \ar[dr,  shorten <= -0.2cm, 
				shorten >= -0.2cm, "\eta_{A}"'] \ar[drrr, shorten >= -0.2cm, 
				bend left, "1"] & 
				\phantom{1}  &
				&  & B
				\\
				&  TA  \ar[rr, "a"'] & \phantom{1} & A 
				\ar[ur, shorten <= -0.2cm, 
				shorten >= -0.2cm, "f"'] &
			\end{tikzcd}.
		\end{equation*}
		
		If $\overline{f}$ is invertible, then it is called a \emph{pseudo 
			$T$-morphism}; if $\overline{f}$ is identity, then it is called a 
		\emph{strict $T$-morphism}.
		
		A \emph{colax $T$-morphism} is defined in the same way except that the 
		$2$-cell $\overline{f}$ is reversed in direction.
	\end{defi}

	\begin{defi}
		\label{def:T-trans}
		Let $(f, \overline{f})$ and $(g, \overline{g})$ be lax $T$-morphisms 
		from $(A, a)$ to $(B, b)$ as defined in 
		\longref{Definition}{def:T-mor}. A 
		\emph{$T$-transformation} $\rho \colon (f, \overline{f}) \to (g, 
		\overline{g})$ is a $2$-cell $\rho \colon f \Rightarrow g$ in $\twoA$, 
		satisfying
		\begin{equation*}
			\begin{tikzcd}
				TA \ar[d, "a"'] \ar[r, bend left, "Tf"{name=T}] \ar[r, bend 
				right, 
				"Tg"'{name=M}] \ar[Rightarrow, shorten=2mm, from=T, to=M, 
				"T\rho"] & TB 
				\ar[d, "b"] 
				\\
				A \ar[r, bend right, "g"'] \ar[Rightarrow,
				shorten=5mm,  start 
				anchor={[yshift=-0.25cm]}, end 
				anchor={[yshift=-0.25cm]},
				from=1-2, "\overline{g}"]& B
			\end{tikzcd}
			=
			\begin{tikzcd}
				TA \ar[d, "a"'] \ar[r, bend left, "Tf"] & TB \ar[d, "b"] 
				\\
				A \ar[r, bend right, "g"'{name=B}] \ar[r, bend left, 
				"f"{name=M}] \ar[Rightarrow, shorten=2mm, from=M, to=B, 
				"\rho"] \ar[Rightarrow,
				shorten=5mm,  start 
				anchor={[yshift=0.25cm]}, end 
				anchor={[yshift=0.25cm]},
				from=1-2, "\overline{f}"'] & B
			\end{tikzcd}.
		\end{equation*}
	\end{defi}
	
	\begin{nota}
		We write $\TAlg_l$ for the $2$-category of strict $T$-algebras, lax 
		$T$-morphisms, and $T$-transformations, $\TAlg_p$ for the $2$-category 
		of strict $T$-algebras, pseudo $T$-morphisms, and $T$-transformations, 
		and  $\TAlg_s$ for the $2$-category of strict $T$-algebras, strict 
		$T$-morphisms, and $T$-transformations; whereas for that of colax 
		$T$-morphisms, we write $\TAlg_c$.
	\end{nota}
	
	\subsubsection{Presheaf \texorpdfstring{$2$}{2}-categories}
	\label{subsubsec:presheaves_2-categories}
	
	\begin{nota}
		Let $\twoA, \twoB$ be $2$-categories. We denote the $2$-category of 
		$2$-functors, lax natural 
		transformations, 
		and modifications by $[\twoA, \twoB]_l$.
	\end{nota}
	
	Let $\twoC$ be a $2$-category. It is worth mentioning how the 
	presheaf 
	$2$-category $[\twoC, \Cat]$ corresponds to the 
	$2$-category of (strict) $T$-algebras, strict $T$-morphisms, and 
	$T$-transformations, whereas $[\twoC, \Cat]_l$ corresponds to the 
	$2$-category 
	of (strict) $T$-algebras, lax $T$-morphisms, and 
	$T$-transformations,
	for some $2$-monad $T$, because this is the key motivation to our proposal 
	of codescent objects of marked coherence data in 
	\longref{Section}{subsec:2-equiv}.
	
	Let $E \colon \ob\twoC \to \twoC$ denote the 
	canonical identity-on-object 
	$2$-functor.
	
	There is a standard result:
	
	\begin{lemma}
		\label{lem:Lan_E_as_left_adj}
		The left Kan extension along $E$ is left $2$-adjoint to the 
		pre-composition 
		with $E$:
		\begin{center}
			\begin{tikzcd}[arrows=<-]
				{[\twoC, \Cat]} \ar[r, bend left, "\Lan{E}"] \ar[r, phantom, 
				"\bot"] & {[\ob\twoC, 
					\Cat]} \ar[l, bend left, "E^*"]
			\end{tikzcd}.
		\end{center}
		Moreoever, for any $2$-functor $X \colon \twoC \to \Cat$, there is a 
		formula:
		\begin{equation}
			\label{eqt:Lan_E}
			\Lan[X]{E} = \sum_{d \in \twoC} \twoC(d, -) \times X(d).
		\end{equation} 
	\end{lemma}
	\begin{proof}
		The adjunction follows immediately from the univeral property of the 
		left Kan extension of any functor $\twoC \to \Cat$ along $E$.
	\end{proof}
	
	Now, we define a $2$-monad $T := E^* \Lan{E} \colon [\ob\twoC, 
	\Cat] \to 
	[\ob\twoC, \Cat]$, with multiplication and unit given by
	\begin{center}
		\scalebox{0.9}{\parbox{1.11\linewidth}{%
				\begin{align*}
					(\mu_X)_c \colon \sum_{d_2 \in 
						\twoC} \sum_{d_1 \in 
						\twoC} \sum_{d 
						\in \twoC} 
					\twoC(d_2, c) \times \twoC(d_1, 
					d_2) \times 
					\twoC(d,d_1) \times 
					X(d) &\to \sum_{d \in \twoC} 
					\twoC(d, c) \times X(d),
					\\
					(f \colon d_1 \to c, g \colon d \to 
					d_1, x \in X(d)) 
					&\mapsto (fg 
					\colon d \to c, x \in X(d)),
				\end{align*}
		}}
	\end{center} and
	\begin{align*}
		(\eta_X)_c \colon X(c) &\to \sum_{d \in \twoC} 
		\twoC(d, c) \times 
		X(d),
		\\
		x &\mapsto (1_c, x),
	\end{align*} respectively.
	
	Our later construction of marked codescent objects of marked coherence 
	data depends on the correspondence between the data of lax $T$-morphisms 
	and that of lax natural transformations.
	
	\begin{theorem}[{\cite[Example 6.6]{BKP:1989}}]
		\label{thm:presheaves_cats_are_T-Alg}
		The $2$-category $\TAlg_l$ of (strict) $T$-algebras, lax $T$-morphisms, 
		and $T$-transformations is isomorphic to $[\twoC, \Cat]_l$.
		
		Similarly, the $2$-category $\TAlg_s$ of (strict) $T$-algebras, strict 
		$T$-morphisms, and $T$-transformations is isomorphic to $[\twoC, \Cat]$.
	\end{theorem}
	\begin{proof}
		A strict $T$-algebra is a pair $(A \in [\ob\twoC, \Cat], a \colon TA 
		\to A)$ satisfying
		\begin{center}
			\begin{tikzcd}
				T^2A \ar[r, "Ta"] \ar[d, "\mu_A"'] & TA \ar[d, "a"]
				\\
				TA \ar[r, "a"'] & A
			\end{tikzcd},
			\begin{tikzcd}
				A \ar[r, "\eta_A"] \ar[dr, "1_A"'] & TA \ar[d, "a"]
				\\
				& A
			\end{tikzcd}.
		\end{center}
		Note that $a \colon TA \to A = \sum_{d \in \twoC} \twoC(d, -) \times 
		A(d) 
		\to A$ is a $2$-natural transformation with component $a_c \colon 
		\sum_{d 
			\in \twoC} \twoC(d, c) \times A(d) \to A(c)$, which is simply a 
		functor 
		$a_{c, d} \colon \twoC(d, c) \times A(d) \to A(c)$ for a choice of $d 
		\in 
		\ob\twoC$. Now by the tensor-hom adjunction, this is equivalent to a 
		functor
		\begin{equation*}
			\twoC(d, c) \to \Cat(A(d), A(c)).
		\end{equation*}
		So the data of a structure morphism $a \colon TA \to A$ is equivalent 
		to a 
		functor $A(f) \colon A(d) \to A(c)$ induced by a morphism $f \colon d 
		\to c$
		in $\twoC$, and a natural transformation 
		\begin{align*}
			A(\gamma) \colon A(f_1) &\to A(f_2),
			\\
			(x_1 \mapsto y_1) &\mapsto (x_2 \mapsto y_2),
		\end{align*}
		induced by a $2$-cell $\gamma \colon f_1 \Rightarrow f_2$ in $\twoC$, 
		where 
		$f_i \colon d \to c$ and $x_i \in A(d)$, for $i = 1, 2$. These are 
		precisely the data of a 
		$2$-functor $A \colon \twoC \to \Cat$.
		
		Next, for an object $(f \colon d_1 \to c, g \colon d \to d_1, x \in 
		A(d))$ 
		of $\sum_{d_1 \in 
			\twoC} \sum_{d \in \twoC} 
		\twoC(d_1, c) \times \twoC(d,d_1) \times A(d)$, the commutative 
		diagrams 
		above amount to
		\begin{align*}
			(a \cdot Ta)_c (f, g, x) = a(f, A(g)(x))
			= A(fg)(x),
		\end{align*}
		and similarly, $A(f)(A(g))(\chi) = A(fg)(\chi)$ for a morphism $\chi 
		\colon 
		x_1 \to x_2$ in $A(d)$; and amount to 
		\begin{equation*}
			(a \cdot \eta_A)_c (x) = a(1_c, x) = A(1_c)(x) = 
			1_{A(c)}(x),
		\end{equation*}
		for an object $x$ of $A(c)$, and similarly, $A(1_c)(\chi) = 
		1_{A(c)}(\chi)$ 
		for a morphism $\chi \colon 
		x_1 \to x_2$ in $A(c)$. These are precisely the functoriality and unity 
		of 
		the 
		$2$-functor $A \colon \twoC \to \Cat$.
		
		Therefore, a strict $T$-algebra $(A, a)$ corresponds to a $2$-functor 
		$A 
		\colon \twoC \to \Cat$.
		
		Let $(F \colon A \to B, \overline{F} \colon b \cdot TF \Rightarrow F 
		\cdot 
		a)$ be a lax $T$-morphism. Then, we have the coherence
		\begin{equation}
			\label{thm:eqt:T-mor-1}
			\begin{tikzcd}[row sep = small, column sep = 0.2em]
				& T^2B \ar[rr, "Tb"] \ar[dr, shorten <= -0.2cm, 
				shorten >= -0.2cm, "\mu_B"'] & \phantom{1} & TB \ar[dr, shorten 
				<= -0.2cm, 
				shorten >= -0.2cm, "b"] & 
				\\
				T^2A \ar[ur, shorten <= -0.2cm, 
				shorten >= -0.2cm, "T^2F"] \ar[dr, shorten <= -0.2cm, 
				shorten >= -0.2cm, "\mu_{A}"'] & & TB \ar[rr, "b"] & 
				\phantom{1} \ar[dl, phantom, "\Downarrow\overline{F}"] & B
				\\
				& TA \ar[ur, shorten <= -0.2cm, 
				shorten >= -0.2cm, "TF"] \ar[rr, "a"'] & \phantom{1} & A 
				\ar[ur, shorten <= -0.2cm, 
				shorten >= -0.2cm,
				"F"']&
			\end{tikzcd}
			=
			\begin{tikzcd}[row sep = small, column sep = 0.2em]
				& T^2B \ar[rr, "Tb"] & \phantom{1} \ar[dl, phantom, "\Downarrow 
				T\overline{F}"] & TB \ar[dr, shorten <= -0.2cm, 
				shorten >= -0.2cm, "b"] \ar[dd, phantom, 
				"\Downarrow\overline{F}"] & 
				\\
				T^2A \ar[ur, shorten <= -0.2cm, shorten >= -0.2cm, "T^2F"] 
				\ar[rr, "Ta"'] \ar[dr, shorten <= -0.2cm, shorten >= 
				-0.2cm, 
				"\mu_{TA}"']& \phantom{1} 
				& TA 
				\ar[ur, shorten <= -0.2cm, shorten >= -0.2cm, "TF"] \ar[dr, 
				shorten <= -0.2cm, shorten >= -0.2cm, "a"'] & & B
				\\
				& TA \ar[rr, "a"'] & \phantom{1} & A \ar[ur, shorten <= -0.2cm, 
				shorten >= -0.2cm, "F"']&
			\end{tikzcd},
		\end{equation}
		\begin{equation}
			\label{thm:eqt:T-mor-2}
			\begin{tikzcd}[row sep = tiny, column sep = tiny]
				& B \ar[dr,  shorten <= -0.2cm, 
				shorten >= -0.2cm,  "\eta_B"'] \ar[drrr, shorten >= -0.2cm, 
				bend left, "1"]&  & 
				\phantom{1}  &
				\\
				A \ar[ur,  shorten <= -0.2cm, 
				shorten >= -0.2cm, "F"] \ar[dr,  shorten <= -0.2cm, 
				shorten >= -0.2cm, "\eta_{A}"'] &  & TB \ar[rr, "b"]& 
				\phantom{1}\ar[dl, phantom, "\Downarrow\overline{F}"] & B
				\\
				&  TA \ar[ur, shorten <= -0.2cm, 
				shorten >= -0.2cm,  "TF"] \ar[rr, "a"'] & \phantom{1} & A 
				\ar[ur, shorten <= -0.2cm, 
				shorten >= -0.2cm, "F"'] &
			\end{tikzcd}
			=
			\begin{tikzcd}[row sep = tiny, column sep = tiny]
				& B  \ar[drrr, shorten >= -0.2cm, bend left, "1"]&  & 
				\phantom{1}  &
				\\
				A \ar[ur,  shorten <= -0.2cm, 
				shorten >= -0.2cm, "F"] \ar[dr,  shorten <= -0.2cm, 
				shorten >= -0.2cm, "\eta_{A}"'] \ar[drrr, shorten >= -0.2cm, 
				bend left, "1"] & 
				\phantom{1}  &
				&  & B
				\\
				&  TA  \ar[rr, "a"'] & \phantom{1} & A 
				\ar[ur, shorten <= -0.2cm, 
				shorten >= -0.2cm, "F"'] &
			\end{tikzcd}.
		\end{equation}
		Note that the modification $\overline{F} \colon b\cdot TF \Rightarrow F 
		\cdot a$ has its component at $c \in \ob\twoC$ being a natural 
		transformation $\overline{F}_c \colon (b\cdot TF)_c \to (F \cdot a)_c$, 
		which means we have
		\begin{center}
			\begin{tikzcd}
				{\sum_{d \in \twoC} \twoC(d, c) \times A(d)} \ar[r, "a_c"] 
				\ar[d, 
				"(TF)_c"'] & A(c) \ar[d, "F(c)"]
				\\
				{\sum_{d \in \twoC} \twoC(d, c) \times B(d)} \ar[r, "b_c"'] 
				\ar[Rightarrow, to=1-2, shorten=8mm, "\overline{F}_c"']& B(c)
			\end{tikzcd},
		\end{center}
		which is equivalent to the following natural transformation 
		$\overline{F}_{c, d}$ for a chosen $d \in \ob\twoC$
		\begin{center}
			\begin{tikzcd}
				{\twoC(d, c) \times A(d)} \ar[r, "a_{c, d}"] \ar[d, 
				"{\twoC(d, c) \times F(d)}"'] & A(c) \ar[d, "F(c)"]
				\\
				{\twoC(d, c) \times B(d)} \ar[r, "b_{c, d}"'] \ar[Rightarrow, 
				to=1-2, 
				shorten=6.25mm,  "\overline{F}_{c, d}"']& B(c)
			\end{tikzcd}.
		\end{center}
		Let $(f, x)$ be an object of $\twoC(d, c) \times A(d)$. The above  diagram then gives
		\begin{align*}
			B(f)(F(d)(x)) \to F(c)(A(f)(x)).
		\end{align*}
		So the component ${\overline{F}_{c, d}}_{(f, x \in A(d))} \colon b_{c, 
			d} 
		\cdot 
		(\twoC(d, c) \times F(d)) (f, x) \to F(c) \cdot a_{c, d} (f, x)$ 
		amounts to 
		the component $B(f)(F(d)(x)) \to F(c)(A(f)(x))$ of a natural 
		transformation 
		\begin{center}
			\begin{tikzcd}
				A(d) \ar[r, "F(d)"] \ar[d, "A(f)"'] & B(d) \ar[d, "B(f)"]
				\\
				A(c) \ar[r, "F(c)"'] \ar[Leftarrow, to=1-2, 
				shorten=4mm, "{{{\overline{F}}_{c, 
							d, f}}}"'] & 
				B(c)
			\end{tikzcd},
		\end{center}
		and this is precisely the data of a lax natural transformation $F 
		\colon A 
		\to B$.
		
		From \eqtref{Equation}{thm:eqt:T-mor-2}, we obtain $\overline{F}_{c, c, 
			1_c} = 1$ as $(a \cdot \eta_A)_c = A(1_c)$. Now, for an object $(f 
		\colon 
		d\to 
		c, g\colon e \to d, y \in A(e))$ of $\twoC(d, c) \times \twoC(e, d) 
		\times 
		A(e)$, we have
		\begin{align*}
			{(\overline{F} \cdot \mu_A)_{c, d, e}}_{(f, g, y)} 	= 
			{\overline{F}_{c, e, fg}}_{(y)}, \quad
			{((\overline{F} \cdot Ta) \circ (b \cdot T\overline{F}))_{c, d, 
					e}}_{(f, g, y)} = (\overline{F}_{c, d, f} \circ A(g))_y 
			\circ (B(f) \circ 
			\overline{F}_{d, e, g})_y.
		\end{align*}
		By \eqtref{Equation}{thm:eqt:T-mor-1}, we obtain $\overline{F}_{c, e, 
			fg} 
		= (\overline{F}_{c, d, f} \circ A(g)) \circ (B(f) \circ 
		\overline{F}_{d, e, 
			g})$. Altogether, we recover the lax unity and the lax naturality 
		of $F$.
		
		Hence, a lax $T$-morphism $(F, \overline{F})$ corresponds to a lax 
		natural 
		transformation $F \colon A \to B$. If in addition $(F, \overline{F})$ 
		is  
		strict, then $\overline{F} = 1$ and hence it corresponds to a 
		$2$-natural 
		transformation $F \colon A \to B$.
	\end{proof}
	
	%	We now have sufficient knowledge to go into the notion of marked limits, 
	%	the interested readers may skip directly to \longref{Section}{sec:marked}.
	%	
	%	\subsubsection{Factorisation systems}
	%	
	%	We recall some basics of factorisation systems, which will be useful in 
	%	establishing later results. A detailed reference can be found in 
	%	\longcite{book:Borceux:1994}.
	%	
	%	\begin{defi}
		%		An \emph{orthogonal factorisation system} on a category $C$ is a pair 
		%		$(\mathcal{E}, \mathcal{M})$ of classes of morphisms in $C$ such that
		%		\begin{enumerate}
			%			\item [$\bullet$] every morphism $f \colon x \to y$ in $C$ can be 
			%			factored as $f = m e$, where $m \in \mathcal{M}$ and $e \in 
			%			\mathcal{E}$;
			%			\item [$\bullet$] for any commutative square 
			%			\begin{center}
				%				\begin{tikzcd}
					%					a \ar[r, "f"] \ar[d, "e"'] & c \ar[d, "m"]
					%					\\
					%					b \ar[r, "g"'] \ar[ur, dashed, "\exists! h"'] & d
					%				\end{tikzcd}
				%			\end{center}
			%		in $C$, where $m \in \mathcal{M}$ and $e \in \mathcal{E}$, there exists 
			%		a unique morphism $h \colon b \to c$ in $C$, making the whole diagram 
			%		commute.
			%		\end{enumerate}
		%	\end{defi}

	\section{Marked limits and their equivalence to 
		\texorpdfstring{$\Cat$}{Cat}-weighted limits}
	\label{sec:marked}
	
	\subsection{Marked limits}
	
	\begin{defi}
		\label{def:marked_cat}
		Let $\twoA$ be a $2$-category. Let $\Sigma$ be a class of morphisms in 
		$\twoA$, which contains all the 
		identities and is closed under composition. The pair $(\twoA, \Sigma)$ 
		is called a \emph{marked $2$-category}.
	\end{defi}

	\begin{defi}[{\cite[I,2, p.14, (i)]{book:Gray:1974}}]
		\label{def:sigmatrans}
		Let $(\twoA, \Sigma)$ be a marked $2$-category. Let $F, G \colon 
		\twoA \rightrightarrows 
		\twoB$ be $2$-functors. A \emph{marked-lax natural transformation} 
		$\alpha \colon F \to G$ between $F$ and $G$ is a lax natural 
		transformation $\alpha \colon F \to G$ such that for any $f \in 
		\Sigma$, the $2$-component $\alpha_f = 1$. 
	\end{defi}
	
	\begin{rk}
		It is important to point out that this notion is more general than the 
		lax normal natural transformation considered in \longcite{Mesiti:2023}, 
		even though any $2$-category $\twoA$ can be 
		recovered by considering the Grothendieck construction of $\triangle(\mathbf{1}) 
		\colon \twoA \to \Cat$. We allow 
		\emph{different} classes $\Sigma$ of morphisms in $\twoA$, as opposed 
		to fixing the specific class $\Sigma = \{(f, 
		1)\}$ of morphisms in $\int^\op \triangle(\mathbf{1})$. 
	\end{rk}
	
	\begin{nota}
		We denote by $A \nrightarrow B$ a morphism from $A$ to $B$ in $\Sigma$.
	\end{nota}
	
	\begin{rk}
		Similarly, we can talk about \emph{marked-pseudo} or 
		\emph{marked-colax} 
		natural transformations.
	\end{rk}
	
	\begin{nota}
		We denote the $2$-category of $2$-functors $\twoA \to \twoB$, 
		marked-lax natural 
		transformations, 
		and modifications by $[\twoA, \twoB]_{l, \Sigma}$. 
	\end{nota}
	
	\begin{nota}
		We denote by 
		$\Sigma(a, b)$ the full subcategory of $\twoA(a, b)$ consisting of 
		morphisms in $\Sigma$
		that have fixed source $a$ and fixed target $b$.
	\end{nota}

	\begin{defi}[{\cite[I,7, p.23, (iii)]{book:Gray:1974}, \cite[Definition 
			2.4]{Szyld:2019}}]
		Let $(\twoA, \Sigma)$ be a small marked $2$-category, i.e., $\ob\twoA$ 
		is small, and let $\twoB$ be a $2$-category. The \emph{marked-lax 
			limit} 
		$\sigmalim{l}{F}$ of a $2$-functor $F \colon \twoA \to 
		\twoB$ is 
		characterised by an isomorphism
		\begin{equation}
			\label{def:sigmalim}
			\twoB(B, \sigmalim{l}{F}) \cong [\twoA, \twoB]_{l, 
				\Sigma}(\triangle(B), 
			F)
		\end{equation}
		in $\Cat$, which is natural in $B \in \ob\twoB$.
	\end{defi}
	
	\begin{rk}
		We can talk about marked-colax or marked-pseudo limits, by replacing 
		the marked-lax natural transformations with marked-colax or 
		marked-pseudo natural transformations, respectively.
	\end{rk}
	
	\subsection{Equivalence to \texorpdfstring{$\Cat$}{Cat}-weighted limits}
	\label{subsec:2-equiv}
	
	We now begin to show that a marked limit can be equivalently expressed as a 
	$\Cat$-weighted limit, which has first been shown in 
	\cite[Theorem 14]{Street:1976}, and we are going to provide a more detailed 
	proof 
	of the result.
	
	To achieve our goal, we introduce the concept of marked codescent objects 
	of 
	marked coherence data, which is a modified version of lax codescent objects 
	of 
	strict coherence data studied in \longcite{Lack:2002}.
	
	Let $(A, a)$ and $(B, b)$ be strict $T$-algebras for a $2$-monad $T$. In 
	\longcite{Lack:2002}, Lack has shown that a lax $T$-morphism is equivalent 
	to a 
	lax codescent cocone of the strict coherence data
	\begin{center}
		\begin{tikzcd}[column sep=large]
			T^3A \ar[r, shift left=2ex, "\mu_{TA}"] \ar[r, shift 
			right=2ex, "T^2a"'] \ar[r, "\scriptstyle{T\mu_A}" description] 
			& T^2A \ar[r, shift left=2ex, "\mu_A"] \ar[r, shift 
			right=2ex, "Ta"' 
			]  &   
			TA   \ar[l,  "\scriptstyle{T\eta_A}" 
			description] 
		\end{tikzcd}.
	\end{center}
	
	Now, let $\twoC$ be an arbitrary $2$-category, and $E \colon \ob\twoC \to 
	\twoC$ be 
	the canonical identity-on-object $2$-functor. Define $T := E^* \Lan{E} 
	\colon [\ob\twoC, \Cat] \to 
	[\ob\twoC, \Cat]$. Then $T$ is a $2$-monad on $[\ob\twoC, \Cat]$, and from 
	\longref{Theorem}{thm:presheaves_cats_are_T-Alg}, a lax 
	$T$-morphism in this case is a lax natural transformation in 
	$[\twoC,\Cat]$. We now investigate 
	how the corresponding lax codescent cocone of 
	the strict coherence data given above looks like concretely when $T 
	= E^* \Lan{E}$.
	
	The lax codescent cocone $(B, G := b \cdot TF \colon TA \to B, \overline{G} 
	:= b 
	\cdot T\overline{F} \colon G \cdot 
	\mu_A \Rightarrow G \cdot Ta)$
	\begin{center}
		\begin{tikzcd}[row sep = tiny, column sep = tiny]
			& TA \ar[dr, shorten <= -0.2cm, shorten >= -0.2cm, "G"] 
			\ar[dd, phantom, "\Downarrow \overline{G}"]&
			\\
			T^2A \ar[ur, shorten <= -0.2cm, shorten >= -0.2cm, 
			"\mu_A"] \ar[dr, shorten <= -0.2cm, shorten >= -0.2cm, 
			"Ta"'] & & B
			\\
			& TA \ar[ur, shorten <= -0.2cm, shorten >= -0.2cm, "G"'] &
		\end{tikzcd}
	\end{center}
	has its component at $c \in \ob\twoC$ given by
	\begin{center}
		\begin{tikzcd}[row sep = tiny, column sep = tiny]
			& {\sum_{d \in \twoC} \twoC(d, c) \times A(d)} \ar[dr, shorten <= 
			-0.2cm, shorten >= -0.2cm, "G(c)"] 
			&
			\\
			{\sum_{d_1 \in \twoC} \sum_{d \in \twoC} \twoC(d_1, c) \times 
				\twoC(d, d_1) \times A(d)} \ar[ur, shorten <= -0.2cm, shorten 
			>= 
			-0.2cm, 
			"{\mu_A}_c"] \ar[dr, shorten <= -0.2cm, shorten >= -0.2cm, 
			"Ta_c"'] & & B(c)
			\\
			& {\sum_{d \in \twoC} \twoC(d, c) \times A(d)}  \ar[ur, shorten <= 
			-0.2cm, shorten >= -0.2cm, "G(c)"'] \ar[Rightarrow, shorten=3mm, 
			from=1-2, 
			"\overline{G}_c"] &
		\end{tikzcd}
	\end{center}
	in our case $T = E^* \Lan{E}$. The component of $\overline{G}_c$ at $(h 
	\colon d_1 \to c, k \colon d \to d_1, x \in A(d))$ is given by
	\begin{align*}
		{\overline{G}_{c, d_1, d}}_{(h, k, x)} = {(B(h) \cdot {F}_k)_{c, d_1, 
				d}}_{(h, k, x)} \colon 
		B(hk)(F_d(x)) \to 
		B(h)(F_{d_1}(A(k)(x))).
	\end{align*}
	So the component of $\overline{G}_c$ at $(h, k, x)$  corresponds to $B(h) 
	\cdot F_k$ in 
	the following 
	diagram of $2$-cells:
	\begin{equation}
		\begin{tikzcd}
			\label{diag:G_bar}
			{A(d)} \ar[r,  "F_d"] \ar[d,  "A(k)"']& 
			{B(d)} \ar[d, 
			"B(k)"] 
			\\
			{A(d_1)} \ar[r, "F_{d_1}"'] \ar[d, 
			"A(h)"'] 			\ar[Leftarrow, to=1-2, 
			shorten=4.5mm, "{F}_k"'] & 
			{B(d_1)} \ar[d, 
			"B(h)"] 
			\\
			A(c) \ar[r, "F_c"'] \ar[Leftarrow, to=2-2, 
			shorten=4.5mm, "{F}_h"'] & {B(c)}
		\end{tikzcd}.
	\end{equation}
	
	Now let $F \colon A \to B$ be a marked-lax 
	natural transformation, then for any morphism $k$ in $\Sigma \subseteq 
	\twoC$, $F_k = 1$. So from 
	\longref{Diagram}{diag:G_bar}, we then have ${\overline{G}_c}_{(h, k, x)} = 
	B(h) \cdot F_k = 1$. We may express this situation as follows:
	\begin{center}
		\begin{tikzcd}[scale cd = 0.9 ,row sep = tiny, column sep = small]
			& & {\sum_{d \in \twoC} \twoC(d, c) \times A(d)} \ar[dr, "G(c)"] 
			&
			\\
			{\sum_{d_1 \in \twoC} \sum_{d \in \twoC} \twoC(d_1, c) \times 
				\Sigma(d, d_1) \times A(d)} \ar[r, hook, "\inc"]  & {\sum_{d_1 
					\in 
					\twoC} \sum_{d \in \twoC} \twoC(d_1, c) \times \twoC(d, 
				d_1) \times 
				A(d)} \ar[ur, "{\mu_A}_c"] \ar[dr, "Ta_c"']& & B(c)
			\\
			& & {\sum_{d \in \twoC} \twoC(d, c) \times A(d)} \ar[ur, "G(c)"'] 
			\ar[Rightarrow, shorten=2.5mm, 
			from=1-3, 
			"\overline{G}_c"]  &
		\end{tikzcd}
	\end{center}
	\begin{equation}
		\label{eqt:sigma_modified}
		= 1.
	\end{equation}
	
	Since all identities are included in $\Sigma$, there is a factorisation
	\begin{equation}
		\label{diag:sigma_factorisation}
		\begin{tikzcd}[scale cd = 0.9, row sep = normal, column sep = tiny]
			{\sum_{d \in \twoC} \twoC(d, c) \times A(d)} \ar[rr, hook] \ar[dr, 
			hook] & & {\sum_{d_1 \in 
					\twoC} \sum_{d \in \twoC} \twoC(d_1, c) \times \twoC(d, 
				d_1) \times 
				A(d)}
			\\
			&{\sum_{d_1 \in \twoC} \sum_{d \in \twoC} \twoC(d_1, c) \times 
				\Sigma(d, d_1) \times A(d)} \ar[ur, hook] &
		\end{tikzcd}.
	\end{equation}
	
	\longref{Diagram}{diag:sigma_factorisation} motivates an abstract diagram
	\begin{center}
		\begin{tikzcd}
			X_0 \ar[rr, hook] \ar[dr, hook] &  & X_1
			\\
			& X_{\sigma} \ar[ur, hook]&
		\end{tikzcd}
	\end{center}
	in $\Cat$, whereas \eqtref{Equation}{eqt:sigma_modified} motivates a 
	coherence 
	condition
	\begin{equation*}
		\begin{tikzcd}[row sep = tiny, column sep = small]
			& & X_0 \ar[dr, shorten <= -0.2cm, 
			shorten >= -0.2cm, "y"] \ar[dd, phantom, "\Downarrow\upsilon"] &
			\\
			{X_\sigma} \ar[r, "\frz"] & X_1 \ar[ur, shorten <= -0.2cm, 
			shorten >= -0.2cm, "\frs"] \ar[dr, shorten <= -0.2cm, 
			shorten >= -0.2cm, "\frt"']& & Y
			\\
			& & X_0 \ar[ur, shorten <= -0.2cm, 
			shorten >= -0.2cm, "y"'] &
		\end{tikzcd}
		=
		\begin{tikzcd}
			{X_\sigma} \ar[rr, "y_\sigma"]& & Y
		\end{tikzcd}
	\end{equation*}
	where $X_\sigma$, $\frz \colon X_\sigma \to X_1$, $y \colon X_0 \to y$ , 
	and 
	$y_\sigma \colon 
	X_\sigma \to X$ are still left to be defined.
	
	\begin{defi}
		\label{def:sigma_simplex}
		The category $\Delta_\sigma$ is generated by a set $\{[0], [1], [2], 
		[\sigma]\}$ of objects and a set $\{p, m, q, s, t, i, j, k\}$ of 
		morphisms, depicted as follows:
		\begin{center}
			\begin{tikzcd}
				{[2]} & {[1]} \ar[l, shift left=2ex, "q"] \ar[l, shift 
				right=2ex, "p"'] \ar[l, "m" description] \ar[rr, shift 
				right=2ex, "i"' 
				] \ar[dr, shorten <= 0.3cm, 
				shorten >= -0.2cm, "k"'] & & {[0]}  \ar[ll,  
				"t" description] \ar[ll, 
				shift 
				right=2ex, "s"']
				\\
				&  &{[\sigma]} \ar[ur, shorten <= -0.2cm, 
				shorten >= 0.3cm, "j"'] &
			\end{tikzcd},
		\end{center}
		with relations given by 
		\begin{equation*}
			{is} = {it} = 1, {ps} = {ms},
			{qt} = {mt}, {pt} = 
			{qs},  \text{and} 
			\hspace{0.5em} i = jk,
		\end{equation*}
		which are exactly the simplicial relations with an extra equation $i 
		= jk$. We call these the \emph{marked relations}.
	\end{defi}
	
	\begin{defi}
		\label{def:sigma_coh_data}
		A collection of \emph{marked coherence data} $\frX_\sigma$ is a 
		$2$-functor $\frX_\sigma \colon {\Delta_\sigma}^\op \to 
		\twoC$. This means $\frX_\sigma$ is a diagram
		\begin{center}
			\begin{tikzcd}
				X_2 \ar[r, shift left=2ex, "\frp"] \ar[r, shift 
				right=2ex, "\frq"'] \ar[r, "\frm" description] 
				& X_1 \ar[rr, shift left=2ex, "\frs"] \ar[rr,  "\frt"' 
				description]   &   & 
				X_0  \ar[ll, shift 
				left=2ex, "\fri" 
				] \ar[dl, shorten <= 0.3cm, 
				shorten >= -0.2cm,  "\frj"]
				\\
				& & {X_\sigma} \ar[ul, shorten <= -0.2cm, 
				shorten >= 0.3cm,  "\frz"]
			\end{tikzcd}
		\end{center}
		in $\twoC$, that satisfies the \emph{marked identities}
		\begin{equation*}
			\mathfrak{si} = \mathfrak{ti} = 1, \mathfrak{sp} = \mathfrak{sm},
			\mathfrak{tq} = \mathfrak{tm}, \mathfrak{tp} = 
			\mathfrak{sq},  \text{and} 
			\hspace{0.5em} \fri = \frj\frz.
		\end{equation*}
	\end{defi}
	
	\begin{defi}
		\label{def:sigma_codescent_cocone}
		A \emph{marked codescent cocone} of the marked coherence data 
		$\frX_\sigma$ 
		defined in 
		\longref{Definition}{def:sigma_coh_data} consists of a triple $(Y 
		\in \ob\twoC, y_\sigma \colon X_\sigma \to Y, \upsilon \colon ys 
		\Rightarrow yt)$:
		\begin{center}
			\begin{tikzcd}[row sep = tiny, column sep = tiny]
				& X_0 \ar[dr, shorten <= -0.2cm, shorten >= -0.2cm, "y := 
				y_\sigma \frj"] 
				\ar[dd, phantom, "\Downarrow \upsilon"]&
				\\
				X_1 \ar[ur, shorten <= -0.2cm, shorten >= -0.2cm, 
				"\frs"] \ar[dr, shorten <= -0.2cm, shorten >= -0.2cm, 
				"\frt"'] & & Y
				\\
				& X_0 \ar[ur, shorten <= -0.2cm, shorten >= -0.2cm, "y := 
				y_\sigma \frj"'] &
			\end{tikzcd}
		\end{center}
		where $y := y_\sigma \frj \colon X_0 \to X$, that satisfies the 
		\emph{multiplicative equation}
		\begin{equation*}
			\begin{tikzcd}[row sep = small, column sep = tiny]
				& X_1 \ar[rr, "\frs"] & & X_0 \ar[dr, shorten <= -0.2cm, 
				shorten >= -0.2cm, "y"] \ar[dd, phantom, "\Downarrow \upsilon"] 
				& 
				\\
				X_2 \ar[ur, shorten <= -0.2cm, shorten >= -0.2cm, "\frp"] 
				\ar[rr, "\frm"'] \ar[dr, shorten <= -0.2cm, shorten >= -0.2cm, 
				"\frq"']& & X_1 
				\ar[ur, shorten <= -0.2cm, shorten >= -0.2cm, "\frs"] \ar[dr, 
				shorten <= -0.2cm, shorten >= -0.2cm, "\frt"'] & & Y
				\\
				& X_1 \ar[rr, "\frt"'] & & X_0 \ar[ur, shorten <= -0.2cm, 
				shorten >= -0.2cm, "y"']&
			\end{tikzcd}
			=
			\begin{tikzcd}[row sep = small, column sep = tiny]
				& X_1 \ar[rr, "\frs"] \ar[dr, shorten <= -0.2cm, 
				shorten >= -0.2cm, "\frt"'] & \phantom{1}\ar[dr, 
				phantom,  
				"\Downarrow\upsilon"] & X_0 \ar[dr, shorten <= -0.2cm, 
				shorten >= -0.2cm, "y"] & 
				\\
				X_2 \ar[ur, shorten <= -0.2cm, 
				shorten >= -0.2cm, "\frp"] \ar[dr, shorten <= -0.2cm, 
				shorten >= -0.2cm, "\frq"'] & & X_0 \ar[rr, "y"] & 
				\phantom{1} \ar[dl, phantom, "\Downarrow\upsilon"] & Y
				\\
				& X_1 \ar[ur, shorten <= -0.2cm, 
				shorten >= -0.2cm, "\frs"] \ar[rr, "\frt"'] & \phantom{1} & X_0 
				\ar[ur, shorten <= -0.2cm, 
				shorten >= -0.2cm,
				"y"']&
			\end{tikzcd}
		\end{equation*} 
		and the \emph{marked equation}
		\begin{equation*}
			\begin{tikzcd}[row sep = tiny, column sep = small]
				& & X_0 \ar[dr, shorten <= -0.2cm, 
				shorten >= -0.2cm, "y"] \ar[dd, phantom, "\Downarrow\upsilon"] &
				\\
				X_\sigma \ar[r, "\frz"]	& X_1 \ar[ur, shorten <= -0.2cm, 
				shorten >= -0.2cm, "\frs"] \ar[dr, shorten <= -0.2cm, 
				shorten >= -0.2cm, "\frt"']& & Y
				\\
				& & X_0 \ar[ur, shorten <= -0.2cm, 
				shorten >= -0.2cm, "y"'] &
			\end{tikzcd}
			=
			\begin{tikzcd}
				X_\sigma \ar[rr, "y_\sigma"]& & Y
			\end{tikzcd}.
		\end{equation*}
	\end{defi}
	
	\begin{defi}
		\label{def:sigma_codescent_mor}
		Let $(Y_1, {y_\sigma}_1 \colon X_\sigma \to Y_1, \upsilon_1 \colon 
		{y_\sigma}_1 \frj s \Rightarrow {y_\sigma}_1 \frj t)$ and $(Y_2, 
		{y_\sigma}_2 \colon X_\sigma \to Y_2, \upsilon_2 \colon {y_{\sigma}}_2 
		\frj s \Rightarrow {y_{\sigma}}_2 \frj t)$ be two marked codescent 
		cocones.  A \emph{morphism of marked codescent cocones} from $(Y_1, 
		{y_\sigma}_1, 
		\upsilon_1)$ to 
		$(Y_2, {y_{\sigma}}_2, \upsilon_2)$ consists of a morphism $f \colon 
		Y_1 \to Y_2$ and a $2$-cell $\theta 
		\colon f{y_\sigma}_1 \frj
		\Rightarrow {y_{\sigma}}_2 \frj$ in 
		$\twoC$, such that
		\begin{equation*}
			\theta \frt \circ f\upsilon_1 = \upsilon_2 \circ \theta \frs.
		\end{equation*}
	\end{defi}
	
	\begin{rk}
		\label{rk:sigma_cocone_mor}
		In particular, if $Y_1 = Y_2$, then a morphism of lax codescent cocones 
		reduces to a $2$-cell $\theta 
		\colon {y_\sigma}_1 \frj
		\Rightarrow {y_{\sigma}}_2 \frj$ satisfying 
		\begin{equation*}
			\theta \frt \circ \upsilon_1 = \upsilon_2 \circ \theta \frs.
		\end{equation*}
	\end{rk}
	
	\begin{nota}
		We denote by $\SigmaCocone(\frX_\sigma, Y)$ the category of marked 
		codescent cocones of $\frX_\sigma$ with a fixed nadir $Y \in \ob\twoC$ 
		and 
		the morphisms of 
		marked codescent cocones as in \longref{Remark}{rk:sigma_cocone_mor}.
	\end{nota}
	
	\begin{defi}
		\label{def:marked_codescent_obj}
		A \emph{marked codescent object} $Z$ of the marked coherence data in 
		\longref{Definition}{def:sigma_coh_data} is the universal marked 
		codescent 
		cocone, i.e., it is characterised by
		\begin{equation*}
			\SigmaCocone(\frX_\sigma, Y) \cong \twoC(Z, Y),
		\end{equation*}
		natural in $Y$.
	\end{defi}
	
	The following proposition is crucial in establishing our main result in 
	this subsection. It tells us that a marked codescent cocone $(Y, y_\sigma, 
	\upsilon)$ of a marked coherence data 
	$\frX_\sigma$ corresponds bijectively to a $2$-natural transformation from 
	the \emph{marked weight} $W$ to the 
	$2$-presheaf $\twoC(\frX_\sigma -, Y)$; in other words, a marked codescent 
	object is described equivalently as a weighted colimit. This is the main 
	reason that we can 
	'strictify' a marked-lax natural transformation into a $2$-natural 
	transformation.
	
	\begin{nota}
		We view a poset $\{0 < 1 < 2 < \cdots < n\}$ as a category, which is 
		denoted as $\mathbf{n} :=  \{0 \to 1 \to 2 \to \cdots \to n\}$.
	\end{nota}
	
	\begin{pro}
		\label{pro:sigma_cocone_bijection}
		Marked codescent objects can be described by weighted colimits.
		
		More precisely, there exists a weight $W \colon \Delta_\sigma 
		\to 
		\Cat$, which we call the \emph{marked weight}, such that
		\begin{equation*}
			[\Delta_\sigma, \Cat](W, \twoC(\frX_\sigma -, Y)) \cong 
			\SigmaCocone(\frX_\sigma, Y).
		\end{equation*}
	\end{pro}
	
	\begin{proof}
		\label{def:sigma_weight}
		We define the \emph{marked weight} $W$ as follows. Let $W 
		\colon \Delta_\sigma 
		\to 
		\Cat$ be a $2$-functor, which acts as the usual embedding of 
		the $2$-truncated simplex 
		category $\Delta_2$ into $\Cat$ on the objects 
		$\{[0], [1], [2]\}$ and the morphisms $\{p, m, q, s, t, i\}$, 
		and in 
		addition,   
		$W([\sigma]) = \mathbf{0}$, $W(k)  \colon \mathbf{1} \to \mathbf{0}$ is 
		the 
		constant 
		map at 
		$0$, and $W(j) \colon \mathbf{0} \to \mathbf{0}$ is the identity on 
		$\mathbf{0}$. 
		This means 
		$W$ is a diagram
		\begin{center}
			\begin{tikzcd}[column sep = large]
				{\mathbf{2}} & {\mathbf{1}} \ar[l, shift left=2ex, "Wq"] 
				\ar[l, shift 
				right=2ex, "Wp"'] \ar[l, "Wm" description] 
				\ar[rr, shift 
				right=2ex, "Wi"' 
				] \ar[dr, start anchor={[xshift=-0.01cm, yshift=-0.25cm]}, 
				shorten >= 0cm, "Wk"'] & & {\mathbf{0}}  \ar[ll,  
				"Wt" description] \ar[ll, 
				shift 
				right=2ex, "Ws"']
				\\
				&  &{\mathbf{0}} \ar[ur, shorten <= 0cm, 
				end 
				anchor={[xshift=0.01cm, yshift=-0.25cm]}, "Wj"'] &
			\end{tikzcd}
		\end{center}
		in $\Cat$, satisfying the \emph{marked relations}.
		
		A standard 
		checking verifies that any $2$-natural transformation $\gamma \colon W 
		\to \twoC(\frX_\sigma -, Y)$ corresponds to a marked codescent cocone, 
		and 
		a modification $\gamma_1 \to \gamma_2$ is 
		equivalently a morphism from $(Y, y_1, \upsilon_1)$ to $(Y, y_2, 
		\upsilon_2)$.
	\end{proof}
	
	Now, let us consider the weighted colimit $W * \frX_\sigma =: X^\ddagger$.

	\begin{lemma}
		\label{lem:sigma_codescent_cocone}
		There exists a morphism $x^\ddagger \colon X_0 \to X^\ddagger$ and a 
		$2$-cell $\chi^\ddagger \colon x^\ddagger \frs \Rightarrow x^\ddagger 
		\frt$, such that $(X^\ddagger, x^\ddagger, \chi^\ddagger)$ is a marked 
		codescent 
		cocone of the marked coherence data $\frX_\sigma$.
	\end{lemma}
	
	\begin{proof}
		There is an isomorphism of categories
		\begin{equation*}
			\twoC(X^\ddagger, Y) \cong [\Delta_\sigma, \Cat](W, 
			\twoC(\frX_\sigma-, Y))
		\end{equation*}
		natural in $Y \in \ob\twoC$. The result follows directly from 
		\longref{Proposition}{pro:sigma_cocone_bijection}.
	\end{proof}
	
	\begin{pro}
		\label{pro:sigma_codescent_obj}
		The marked codescent cocone $(X^\ddagger, x^\ddagger, \chi^\ddagger)$ 
		constructed in \longref{Lemma}{lem:sigma_codescent_cocone} is the 
		marked codescent object of the marked coherence data $\frX_\sigma$, 
		i.e., 
		it is 
		the 
		universal marked codescent cocone.
	\end{pro}
	
	\begin{proof}
		Suppose that $(B, G \colon X_0 \to B, \overline{G} \colon G\frs \to 
		G\frt)$ 
		is a marked codescent cocone of $\frX_\sigma$.
		
		In the proof of \longref{Lemma}{lem:sigma_codescent_cocone}, we see 
		that a marked codescent cocone of $\frX_\sigma$ is equivalent to a 
		$2$-natural transformation $\beta \colon W \to \twoC(\frX_\sigma -, B)$ 
		with $\beta_{[0]}(0) = G$, $\beta_{[1]}(\iota) = \overline{G}$, and the 
		naturality amounts to the multiplicative and marked equations.
		
		A standard argument shows that a morphism between marked codescent 
		cocones 
		$(B, G_1, 
		\overline{G_1})$ and 
		$(B, G_2, \overline{G_2})$ is equivalent to a modification $\beta_1 \to 
		\beta_2$. And the universal property of the marked codescent cocone 
		$(X^\ddagger, 
		x^\ddagger, \chi^\ddagger)$ follows from that of the weighted colimit. 
	\end{proof}
	
	Recall that we have the $2$-monad $T = E^* 
	\Lan{E} 
	\colon [\ob\twoC, \Cat] \to 
	[\ob\twoC, \Cat]$, and we know that a strict $T$-algebra $(A, a \colon TA 
	\to A)$ is equivalent to a $2$-functor $A \colon \twoC \to \Cat$.
	
	Let $\frA_\sigma \colon {\Delta_\sigma}^\op \to [\twoC, \Cat]$ be a 
	collection of marked coherence data in $[\twoC, \Cat]$ as follows:
	\begin{center}
		\begin{tikzcd}
			T^3A \ar[r, shift left=2ex, "\mu_{TA}"] \ar[r, shift 
			right=2ex, "T^2a"'] \ar[r, "\scriptstyle{T\mu_A}" description] 
			& T^2A \ar[rr, shift left=2ex, "\mu_A"] \ar[rr,  "Ta"' 
			description]   &   & 
			TA  \ar[ll, shift 
			left=2ex, "T\eta_A" 
			] \ar[dl, shorten <= 0.3cm, 
			shorten >= -0.2cm,  "T\eta_A"]
			\\
			& & {A_\sigma} \ar[ul, shorten <= -0.2cm, 
			shorten >= 0.3cm, hook',  "\iota"]
		\end{tikzcd},
	\end{center}
	where $A_\sigma := {\sum_{d_1 \in \twoC} \sum_{d \in \twoC} \twoC(d_1, -) 
		\times \Sigma(d, d_1) \times A(d)}$.
	
	Then, a marked codescent cocone of $\frA_\sigma$ consists of a morphism 
	$G_\sigma \colon A_\sigma \to B$, where $B \colon \twoC \to \Cat$ is a 
	$2$-functor, and a $2$-cell
	\begin{center}
		\begin{tikzcd}[row sep = tiny, column sep = tiny]
			& TA \ar[dr, shorten <= -0.2cm, shorten >= -0.2cm, "G"] 
			\ar[dd, phantom, "\Downarrow \overline{G}"]&
			\\
			T^2A \ar[ur, shorten <= -0.2cm, shorten >= -0.2cm, 
			"\mu_A"] \ar[dr, shorten <= -0.2cm, shorten >= -0.2cm, 
			"Ta"'] & & B
			\\
			& TA \ar[ur, shorten <= -0.2cm, shorten >= -0.2cm, "G"'] &
		\end{tikzcd}
	\end{center}
	in $[\twoC, \Cat]$, where $G := G_\sigma T\eta_A$, satisfying the 
	multiplicative equation
	\begin{equation*}
		\begin{tikzcd}[row sep = small, column sep = tiny]
			& T^2A \ar[rr, "\mu_A"] & & TA \ar[dr, shorten <= -0.2cm, 
			shorten >= -0.2cm, "G"] \ar[dd, phantom, "\Downarrow \overline{G}"] 
			& 
			\\
			T^3A \ar[ur, shorten <= -0.2cm, shorten >= -0.2cm, "\mu_{TA}"] 
			\ar[rr, "T\mu_A"'] \ar[dr, shorten <= -0.2cm, shorten >= -0.2cm, 
			"T^2a"']& & T^2A
			\ar[ur, shorten <= -0.2cm, shorten >= -0.2cm, "\mu_A"] \ar[dr, 
			shorten <= -0.2cm, shorten >= -0.2cm, "Ta"'] & & B
			\\
			& T^2A \ar[rr, "Ta"'] & & TA \ar[ur, shorten <= -0.2cm, 
			shorten >= -0.2cm, "G"']&
		\end{tikzcd}
		=
		\begin{tikzcd}[row sep = small, column sep = tiny]
			& T^2A \ar[rr, "\mu_A"] \ar[dr, shorten <= -0.2cm, 
			shorten >= -0.2cm, "Ta"'] & \phantom{1}\ar[dr, 
			phantom,  
			"\Downarrow\overline{G}"] & TA \ar[dr, shorten <= -0.2cm, 
			shorten >= -0.2cm, "G"] & 
			\\
			T^3A \ar[ur, shorten <= -0.2cm, 
			shorten >= -0.2cm, "\mu_{TA}"] \ar[dr, shorten <= -0.2cm, 
			shorten >= -0.2cm, "T^2a"'] & & T^2A \ar[rr, "G"] & 
			\phantom{1} \ar[dl, phantom, "\Downarrow\overline{G}"] & B
			\\
			& T^2A \ar[ur, shorten <= -0.2cm, 
			shorten >= -0.2cm, "\mu_A"] \ar[rr, "Ta"'] & \phantom{1} & TA 
			\ar[ur, shorten <= -0.2cm, 
			shorten >= -0.2cm,
			"G"']&
		\end{tikzcd}
	\end{equation*} 
	and the marked equation
	\begin{equation*}
		\begin{tikzcd}[row sep = tiny, column sep = small]
			& & TA \ar[dr, shorten <= -0.2cm, 
			shorten >= -0.2cm, "G"] \ar[dd, phantom, "\Downarrow\overline{G}"] &
			\\
			A_\sigma \ar[r, hook, "\iota"]	& T^2A \ar[ur, shorten <= -0.2cm, 
			shorten >= -0.2cm, "\mu_A"] \ar[dr, shorten <= -0.2cm, 
			shorten >= -0.2cm, "Ta"']& & B
			\\
			& & TA \ar[ur, shorten <= -0.2cm, 
			shorten >= -0.2cm, "G"'] &
		\end{tikzcd}
		=
		\begin{tikzcd}
			A_\sigma \ar[rr, "G_\sigma"]& & B
		\end{tikzcd}.
	\end{equation*}
	
	The following proposition justifies our notion of marked codescent cocones 
	and 
	shows that it is really the bridge to marked-lax natural transformations.
	
	\begin{pro}
		\label{pro:sigma_cocone_bijection_with_sigma_trans}
		There is an isomorphism of categories
		\begin{equation*}
			\SigmaCocone(\frA_\sigma, B) \cong [\twoC, \Cat]_{l, \Sigma}(A, B),
		\end{equation*}
		natural in B.
	\end{pro}
	
	\begin{proof}
		In \longcite{Lack:2002}, it is shown that a 
		lax codescent cocone
		\begin{center}
			\begin{tikzcd}[row sep = tiny, column sep = tiny]
				& TA \ar[dr, shorten <= -0.2cm, shorten >= -0.2cm, "G"] 
				\ar[dd, phantom, "\Downarrow \overline{G}"]&
				\\
				T^2A \ar[ur, shorten <= -0.2cm, shorten >= -0.2cm, 
				"\mu_A"] \ar[dr, shorten <= -0.2cm, shorten >= -0.2cm, 
				"Ta"'] & & B
				\\
				& TA \ar[ur, shorten <= -0.2cm, shorten >= -0.2cm, "G"'] &
			\end{tikzcd}
		\end{center}
		is equivalent to a lax $T$-morphism $(F \colon A \to B, \overline{F} 
		\colon b \cdot Tf \Rightarrow F \cdot a)$, in which we have $G = b 
		\cdot TF$ 
		and $\overline{G} = b \cdot T\overline{F}$.
		
		Similarly, we can express a marked codescent cocone $(B, G_\sigma 
		\colon 
		A_\sigma \to B, \overline{G})$ in terms of $(F, \overline{F})$, except 
		that our marked equation replaces the original unital equation. It 
		remains to transform the marked equation into an equation in terms of 
		$(F, \overline{F})$.
		
		To ask
		\begin{equation*}
			\begin{tikzcd}[row sep = tiny, column sep = small]
				& & TA \ar[dr, shorten <= -0.2cm, 
				shorten >= -0.2cm, "G"] \ar[dd, phantom, 
				"\Downarrow\overline{G}"] &
				\\
				A_\sigma \ar[r, hook, "\iota"]	& T^2A \ar[ur, shorten <= 
				-0.2cm, 
				shorten >= -0.2cm, "\mu_A"] \ar[dr, shorten <= -0.2cm, 
				shorten >= -0.2cm, "Ta"']& & B
				\\
				& & TA \ar[ur, shorten <= -0.2cm, 
				shorten >= -0.2cm, "G"'] &
			\end{tikzcd}
			=
			\begin{tikzcd}
				A_\sigma \ar[rr, "G_\sigma"]& & B
			\end{tikzcd}.
		\end{equation*}
		is equivalently to ask
		\begin{equation*}
			\begin{tikzcd}[row sep = tiny, column sep = small]
				& & T^2B \ar[dr, shorten <= -0.2cm, 
				shorten >= -0.2cm, "Tb"] \ar[dd, phantom, 
				"\Downarrow T\overline{F}"] & &
				\\
				A_\sigma \ar[r, hook, "\iota"]	& T^2A \ar[ur, shorten <= 
				-0.2cm, 
				shorten >= -0.2cm, "T^2F"] \ar[dr, shorten <= -0.2cm, 
				shorten >= -0.2cm, "Ta"']& & TB \ar[r, "b"] & B
				\\
				& & TA \ar[ur, shorten <= -0.2cm, 
				shorten >= -0.2cm, "TF"'] & &
			\end{tikzcd}
			=
			\begin{tikzcd}
				A_\sigma \ar[rr, "G_\sigma"]& & B
			\end{tikzcd}.
		\end{equation*}
		The component of the left hand side at $c \in \twoC$ is given by
		\begin{center}
			\begin{tikzcd}[scale cd = 0.9 ,row sep = tiny, column sep = tiny]
				& &  {\twoC(d_1, c) \times \twoC(d, d_1) \times 
					B(d)} \ar[dr, "Tb_c"] 
				& &
				\\
				{ \twoC(d_1, c) \times 
					\Sigma(d, d_1) \times A(d)} \ar[r, hook, "\inc"]  & 
				{\twoC(d_1, c) \times \twoC(d, d_1) \times 
					A(d)} \ar[ur, "{T^2F}_c"] \ar[dr, "Ta_c"']& & {\twoC(d_1, 
					c) \times B(d_1)} \ar[r, 
				"b_c"] & B(c)
				\\
				& & {\twoC(d_1, c) \times A(d_1)} \ar[ur, "TF_c"'] 
				\ar[Rightarrow, shorten=2.5mm, 
				from=1-3, 
				"T\overline{F}_c"]  & &
			\end{tikzcd},
		\end{center}
		hence the equation is precisely to ask for any $(f, e, x) \in \sum_{d_1 
			\in 
			\twoC} \sum_{d \in \twoC} \twoC(d_1, c) \times 
		\Sigma(d, d_1) \times A(d)$, the $2$-cell ${b_c \cdot 
			T\overline{F}_c}_{(f, e, x)} = B(f) \cdot \overline{F}_c$ is the 
		identity, depicted below:
		\begin{equation*}
			\begin{tikzcd}
				A(d) \ar[r, "F_d"] \ar[d, "A(e)"']	& B(d) \ar[d, "B(e)"]  &
				\\
				A(d_1) \ar[r, "F_{d_1}"'] \ar[Leftarrow, to=1-2, 
				shorten=4.5mm, "\overline{F}_e"']	&  B(d_1) \ar[r, "B(f)"] & 
				B(c)
			\end{tikzcd}
			=
			\begin{tikzcd}
				A(d) \ar[r, "F_d"] & B(d) \ar[r, "B(f)"]  & B(c)
			\end{tikzcd},
		\end{equation*}
		which is equivalent to ask that $\overline{F}_e = 1$. All in all, this 
		asks for any $e \in \Sigma$, $\overline{F}_e = 1$, which is exactly the 
		condition for a marked-lax natural transformation.
		
		It is left to show that a morphism of marked codescent cocones 
		corresponds 
		to a modification between two marked-lax natural transformations.
		
		Following \longcite{Lack:2002}, it is clear that a morphism $(B, 
		G_1 
		\colon TA \to B, \overline{G_1} \colon G_1 \cdot \mu_A \Rightarrow G_1 
		\cdot Ta) \to (B, G_2 \colon TA \to B, \overline{G_2} \colon G_2 \cdot 
		\mu_A \Rightarrow G_2 \cdot Ta)$ of marked codescent cocones given by a 
		$2$-cell $\theta \colon G_1 \Rightarrow G_2$ satisfying $\theta \cdot 
		Ta \cdot \overline{G_1} = \overline{G_2} \cdot \theta \cdot \mu_A$, is 
		equivalent to a $T$-transformation $(F_1 \colon A \to B, \overline{F_1} 
		\colon b \cdot TF_1 \Rightarrow F_1 \cdot a) \to (F_2 \colon A \to B, 
		\overline{F_2} \colon b \cdot TF_2 \Rightarrow F_2 \cdot a)$ given by a 
		$2$-cell $\rho \colon F_1 \Rightarrow F_2$ satisfying
		\begin{equation*}
			\begin{tikzcd}
				TA \ar[d, "a"'] \ar[r, bend left, "TF_1"{name=T}] \ar[r, bend 
				right, 
				"TF_2"'{name=M}] \ar[Rightarrow, shorten=2mm, from=T, to=M, 
				"T\rho"]  & TB 
				\ar[d, "b"] 
				\\
				A \ar[r, bend right, "F_2"'] \ar[Rightarrow,
				shorten=5mm,  start 
				anchor={[yshift=-0.25cm]}, end 
				anchor={[yshift=-0.25cm]},
				from=1-2, "\overline{F_2}"]& B
			\end{tikzcd}
			=
			\begin{tikzcd}
				TA \ar[d, "a"'] \ar[r, bend left, "TF_1"] & TB \ar[d, "b"] 
				\\
				A \ar[r, bend right, "F_2"'{name=B}] \ar[r, bend left, 
				"F_1"{name=M}] \ar[Rightarrow, shorten=2mm, from=M, to=B, 
				"\rho"] \ar[Rightarrow,
				shorten=5mm,  start 
				anchor={[yshift=0.25cm]}, end 
				anchor={[yshift=0.25cm]},
				from=1-2, "\overline{F_1}"'] & B
			\end{tikzcd},
		\end{equation*}
		where $G = b \cdot TF$, $\overline{G} = b \cdot T\overline{F}$ and 
		$\theta = b \cdot T\rho$.
		Consider the components at $c \in \twoC$ , the above equation then says
		\begin{equation*}
			\begin{tikzcd}
				{\twoC(d, c) \times A(d)} \ar[d, "a_c"'] \ar[r, bend left, 
				"{TF_1}_c"{name=T}] \ar[r, bend right, 
				"{TF_2}_c"'{name=M}] \ar[Rightarrow, shorten=5mm, from=T, to=M, 
				"T\rho_c"]  & 
				{\twoC(d, c) \times B(d)} 
				\ar[d, "b_c"] 
				\\
				A(c) \ar[r, bend right, "{F_2}_c"'] \ar[Rightarrow,
				shorten=9mm,  start 
				anchor={[yshift=-0.5cm]}, end 
				anchor={[yshift=-0.5cm]},
				from=1-2, "\overline{F_2}_c"]& B(c)
			\end{tikzcd}
			=
			\begin{tikzcd}
				{\twoC(d, c) \times A(d)} \ar[d, "a_c"'] \ar[r, bend left, 
				"{TF_1}_c"] & {\twoC(d, c) \times B(d)}  \ar[d, "b_c"] 
				\\
				A(c) \ar[r, bend right, "{F_2}_c"'{name=B}] \ar[r, bend left, 
				"{F_1}_c"{name=M}] \ar[Rightarrow, shorten=5mm, from=M, to=B, 
				"\rho_c"] \ar[Rightarrow,
				shorten=9mm,  start 
				anchor={[yshift=0.8cm]}, end 
				anchor={[yshift=0.8cm]},
				from=1-2, "\overline{F_1}_c"']  & B(c)
			\end{tikzcd}.
		\end{equation*}
		This means for $(f \colon d \to c, x \in A(d))$,
		\begin{align*}
			\rho_c(A(f)(x)) \circ \overline{F_1}_f &= \overline{F_2}_f \circ 
			B(f)(\rho_d(x)),
		\end{align*}
		which is exactly the modification axiom.
	\end{proof}
	
	We achieve the following desired theorem:
	
	\begin{pro}
		\label{pro:sigma_left_adj}
		There is an isomorphism of categories
		\begin{equation*}
			[\twoC, \Cat](A^\ddagger, B) \cong 
			[\twoC, \Cat]_{l, \Sigma}(A, B),
		\end{equation*}
		natural in B.
		
		In other words, there is a left adjoint $()^\ddagger \colon [\twoA, 
		\Cat]_{l, \Sigma} 
		\to [\twoA, \Cat]$ to the inclusion.
	\end{pro}
	
	\begin{proof}
		By \longref{Proposition}{pro:sigma_cocone_bijection}, 
		\longref{Proposition}{pro:sigma_cocone_bijection_with_sigma_trans}, and 
		the universal property of the weighted colimit $A^\ddagger = W * 
		\frA_\sigma$, we 
		have the following chain of isomorphisms:
		\begin{align*}
			[\twoC, \Cat](W * \frA_\sigma, B)) &\cong [\Delta_\sigma, \Cat](W, 
			[\twoC, \Cat](\frA_\sigma - , B))
			\\
			&\cong \SigmaCocone(\frA_\sigma, B)
			\\
			&\cong [\twoC, \Cat]_{l, \Sigma}(A, B). \qedhere
		\end{align*}
	\end{proof}
	
	\begin{theorem}
		\label{thm:sigma_left_adj}
		Let $(\twoA, \Sigma)$ be a marked $2$-category. Tthe marked-lax limit 
		of a $2$-functor $F \colon \twoA \to 
		\twoB$ has the same universal property as the $\Cat$-weighted limit 
		$\{\triangle(\mathbf{1})^\ddagger, 
		F\}$.
	\end{theorem}
	
	\begin{proof}
		We have
		\begin{align*}
			\twoB(B, \sigmalim{l}{F}) &\cong [\twoA, \twoB]_{l, 
				\Sigma}(\triangle(B), 
			F)
			\\
			&\cong [\twoA, \Cat]_{l, \Sigma}(\triangle(\mathbf{1}), \twoB(B, F-))
			\\
			&\cong [\twoA, \Cat](\triangle(\mathbf{1})^\ddagger, \twoB(B, F-))
			\\
			&\cong \twoB(B, \{\triangle(\mathbf{1})^\ddagger, F\}). \qedhere
		\end{align*}
	\end{proof}
	
	\subsection{\texorpdfstring{$\Cat$}{Cat}-weighted limits as marked limits}
	
	The other direction that any $\Cat$-weighted limit can be equivalently 
	expressed as a marked-lax limit with the same universal property is 
	well-known. This has first been discussed in Street's 
	paper \longcite{Street:1976}, and also in the recent 
	work \longcite{Mesiti:2023}. Nonetheless, we will have to recall the 
	construction for this 
	result in later sections, hence we decide to present the proof.
	
	Let $\twoA, \twoD$ be $2$-categories, let $W \colon \twoD \to \Cat$ be a 
	weight and $R \colon \twoD \to \twoA$ be a $2$-functor. 
	
	The \emph{$2$-category of elements of $W$} consists of
	\begin{enumerate}
		\item [$\bullet$] objects those pairs $(D, \delta)$;
		\item [$\bullet$] morphisms $(d, \omega) \colon (D_1, \delta_1) \to 
		(D_2, \delta_2)$ given by pairs $(d \colon D_1 \to D_2, \omega 
		\colon 
		Wd \delta_1 \to \delta_2)$;
		\item [$\bullet$] $2$-cells $\alpha \colon (d, \omega) \Rightarrow (d', 
		\omega')$ given by the $2$-cells $\alpha \colon d 
		\Rightarrow d'$ in $\twoD$ such that $\omega' W\alpha \delta_1 = 
		\omega$,
	\end{enumerate}
	and the composition of two morphisms $(h \colon D_1 \to D_2, \eta \colon 
	Wh\delta_1 \to \delta_2)$ and $(k \colon D_2 \to D_3, \kappa\colon 
	Wk\delta_2 \to \delta_3)$ is given by 
	\begin{equation*}
		(kh \colon D_1 \to D_3, \kappa \circ 
		(Wk \cdot \eta) \colon WkWh\delta_1 \to \delta_3).
	\end{equation*}
	We also have the projection of $\twoEl{W}$ onto $\twoD$:
	\begin{align*}
		P \colon \twoEl{W} &\to \twoD,
		\\
		(D, \delta) &\mapsto D.
	\end{align*}
	
	\begin{rk}
		Following \longcite{Mesiti:2023} $\twoEl{W}$ is the lax comma object of 
		$\triangle(\mathbf{1})\colon 1 \to \Cat$ and $W \colon \twoD \to \Cat$
		\begin{center}
			\begin{tikzcd}
				\twoEl{W} \ar[r, "P"] \ar[d, "!"'] & \twoD \ar[d, 
				"W"]
				\\
				1 \ar[Rightarrow, to=1-2, 
				shorten=6mm, "\mu"']
				\ar[r, "\triangle(\mathbf{1})"'] & \Cat
			\end{tikzcd}.
		\end{center}
	\end{rk}
	
	We are going to show that there is a certain marked-lax limit possessing 
	the 
	same 
	universal property of the weighted limit $\{W, R\}$.
	
	\begin{pro}
		\label{pro:weighted_to_sigma}
		Let $\Sigma := \{(d \in \mor \twoD, 1)\}$ be the class of 
		morphisms in $\twoEl{W}$. 
		Then
		\begin{equation*}
			[\twoD, \Cat](W, \twoA(A, R-)) \cong [\twoEl{W}, \twoA]_{l, 
				\Sigma}(\triangle(A), RP).
		\end{equation*}
	\end{pro}
	\begin{proof}
		It suffices to show
		\begin{equation*}
			[\twoD, \Cat](W, \twoA(A, R-)) \cong [\twoEl{W}, \Cat]_{l, 
				\Sigma}(\triangle(\mathbf{1}), \twoA(A, RP-)).
		\end{equation*}
		
		Let $\beta \colon W \to \twoA(A, R-)$ be a $2$-natural transformation, 
		we 
		can then form the composition of $\mu$ and $\beta P$
		\begin{center}
			\begin{tikzcd}
				\twoEl{W} \ar[r, "P"] \ar[d, "!"'] & \twoD \ar[d, "W"'] \ar[d, 
				bend 
				left = 80, end 
				anchor={[xshift=-0.15cm, yshift=-0.15cm]}, "{\twoA(A, R-)}", 
				pos=0.6] \ar[d, phantom, bend left 
				, 
				"{\scalebox{1.2}{$\substack{\Rightarrow 
							\\ \beta}$}}"]
				\\
				1 \ar[Rightarrow, to=1-2, 
				shorten=6mm, "\mu"'] \ar[r, "\triangle(\mathbf{1})"'] & 
				\Cat
			\end{tikzcd}.
		\end{center}
		Since $\beta$ is $2$-natural, the $2$-component of the composition 
		$\beta P 
		\circ \mu$ with 
		respect to a morphism in $\Sigma$ is the identity, hence $\beta P \circ 
		\mu$ is indeed a marked-lax natural transformation. Thus, whiskering 
		with 
		$P$ 
		and pre-composing with $\mu$ defines a functor
		\begin{align}
			\label{functor:G}
			G \colon [\twoD, \Cat](W, \twoA(A, R-)) &\to [\twoEl{W}, \Cat]_{l, 
				\Sigma}(\triangle(\mathbf{1}), \twoA(A, RP-)),
			\\
			\beta &\mapsto \beta P \circ \mu, \nonumber
			\\
			\Gamma \colon \beta_1 \to \beta_2 &\mapsto \Gamma P \circ \mu. 
			\nonumber
		\end{align}
		
		Next, we would like to construct the inverse of $G$.
		
		Let $\alpha \colon 
		\triangle(\mathbf{1})! \to \twoA(A, RP-)$ be a marked-lax 
		natural transformation. For any object $D$ of $\twoD$, define a functor
		\begin{align*}
			\alpha_{(\triangle(D), -)} \colon WD &\to \twoA(A, RD),
			\\
			\delta &\mapsto \alpha_{(D, \delta)},
			\\
			\omega \colon \delta_1 \to \delta_2 &\mapsto \alpha_{(1_D, \omega)}.
		\end{align*}
		Since $\alpha$ is marked natural, for a morphism $(d \colon 
		D_1 
		\to D_2, 1 \colon Wd\delta_1 \to Wd\delta_1) \in \Sigma$, $\alpha_{(d, 
			1)} 
		= 1$. This implies that for any $\delta_1 \in WD_1$, $Rd \cdot 
		\alpha_{(D_1, \delta_1)} 
		= \alpha_{(D_2, Wd\delta_1)}$. So we have a commutative 
		diagram
		\begin{center}
			\begin{tikzcd}
				WD_1 \ar[r, "Wd"] \ar[d, "\alpha_{(\triangle(D_1), -)}"'] & 
				WD_2 
				\ar[d, "\alpha_{(\triangle(D_2), -)}"]
				\\
				\twoA(A, RD_1) \ar[r, "(Rd)_*"'] & \twoA(A, RD_2)
			\end{tikzcd}.
		\end{center}
		Now, let $\theta \colon d \Rightarrow d'$ be a $2$-cell in $\twoD$, and 
		let 
		$(d 
		\colon D_1 \to D_2, W\theta \delta_1 \colon Wd\delta_1 \to 
		Wd'\delta_1)$ 
		and $(d' \colon D_1 \to D_2, 1 \colon Wd'\delta_1 \to Wd'\delta_1)$ be 
		morphisms from $(D_1, \delta_1)$ to $(D_2, Wd'\delta_1)$ in 
		$\twoEl{W}$. 
		Then it is clear that $\theta$ is also a $2$-cell $(d, W\theta\delta_1) 
		\Rightarrow (d', 1)$ in $\twoEl{W}$. Since $\alpha$ is natural, we have
		\begin{equation*}
			\alpha_{(d, W\theta\delta_1)} = (R\theta)_* \cdot \alpha_{(D_1, 
				\delta_1)},
		\end{equation*}	
		which gives
		\begin{equation*}
			\alpha_{(\triangle(D_2), -)} \cdot W\theta = (R\theta)_* \cdot 
			\alpha_{(\triangle(D_1), -)}.
		\end{equation*}
		Altogether, we are able to define a $2$-natural transformation	$\beta 
		\colon W \to \twoA(A, R-)$,	whose $1$-component is given by $\beta_D := 
		\alpha_{(\triangle(D), -)}$, and $2$-component is given by a natural 
		transformation $\beta_d$, where its 
		component at $\delta \in WD$ is given by ${\beta_d}_{\delta 
		} = \alpha_{(d, 1)}$, which is the 
		identity 
		because $\alpha$ is marked-lax natural.
		
		Next, let $\Lambda \colon \alpha_1 \to \alpha_2$ be a modification 
		between 
		the marked-lax natural transformations $\alpha_1$ and $\alpha_2$. 
		Consider 
		the $2$-components of the transformations with respect to the morphism 
		$(d 
		\colon D_1 \to D_2, 1 \colon Wd\delta_1 \to Wd\delta_1)$, we get 
		$\Lambda_{(D_2, Wd\delta_1)} = Rd \cdot \Lambda_{(D_1, \delta_1)}$ from 
		the 
		modification axiom. Therefore, if we construct a map $\Theta 
		\colon \beta_1 \to \beta_2$, where $\beta_1$ and $\beta_2$ are the 
		$2$-natural transformations constructed from $\alpha_1$ and $\alpha_2$, 
		by setting its component at $D$ as a natural transformation $\Theta_D 
		:= 
		\xi^D \colon 
		\alpha_{1_{(\triangle(D), -)}} \to \alpha_{2_{(\triangle(D), -)}}$ with 
		component at $\delta \in WD$ given by $\xi^D_{\delta} := \Lambda_{(D, 
			\delta)} \colon \alpha_{1_{D, \delta}} \to \alpha_{2_{D, \delta}}$, 
		we then 
		obtain a modification $\Theta$ which satisties the modification axiom
		\begin{equation*}
			\Theta_{D_2} \cdot Wd = (Rd)_* \cdot \Theta_{D_1}.
		\end{equation*}
		
		The above constructions of $\beta$ and $\Theta$ define a functor
		\begin{align*}
			G' \colon  [\twoEl{W}, \Cat]_{l, 
				\Sigma}(\triangle(\mathbf{1}), \twoA(A, RP-)) &\to [\twoD, \Cat](W, 
			\twoA(A, 
			R-)),
			\\
			\alpha &\mapsto \beta = \{\alpha_{(\triangle(D), -)}\}_{D \in 
				\twoD},
			\\
			\Lambda \colon \alpha_1 \to \alpha_2 &\mapsto \Theta = 
			\{\{\Lambda_{(D, 
				\delta)} \}_{\delta \in WD}\}_{D \in \twoD}.
		\end{align*}
		
		A straightforward checking shows that $G$ and $G'$ are inverses of each 
		other.
	\end{proof}
	
	\begin{theorem}
		Let $\twoA, \twoD$ be $2$-categories. Let $W \colon \twoD \to \Cat$ be 
		a weight and $R \colon \twoD \to 
		\twoA$ be a $2$-functor. Let $\Sigma := \{(d \in \mor \twoD, 1)\}$ be 
		the class of 
		morphisms in the $2$-category of elements $\twoEl{W}$ of $W$, so that 
		$(\twoEl{W}, \Sigma)$ is a marked $2$-category. The 
		$\Cat$-weighted limit $\{W, R\}$ has the same universal property as the 
		marked-lax limit of the $2$-functor $RP \colon \twoEl{W} \to 
		\twoA$.
	\end{theorem}
	
	\begin{proof}
		By \longref{Proposition}{pro:weighted_to_sigma}, we have a chain of 
		isomorphisms
		\begin{align*}
			\twoA(A, \{W, R\}) &\cong [\twoD, \Cat](W, \twoD(A, R-))
			\\
			&\cong [\twoEl{W}, \twoA]_{l, \Sigma}(\triangle(A), RP)
			\\
			&\cong \twoA(A, \sigmalim{l}{RP}). \qedhere
		\end{align*}
	\end{proof}
	
	\section{Preliminaries on enhanced \texorpdfstring{$2$}{2}-category theory}
	\label{sec:F}
	
	\subsection{\texorpdfstring{$\F$}{F}-categories}
	\label{subsec:F-cats}
	
	We recall the basics of enhanced $2$-category theory, which was first 
	proposed by Lack and Shulman in 
	\longcite{LS:2012}. A gentle introduction can also be found in
	\longcite{Bourke:2014} by Bourke.
	
	\begin{defi}
		Let $\F$ be the full subcategory of the arrow category 
		$\Cat^\mathbf{2}$ 
		of the category
		$\Cat$, determined by the fully faithful and injective-on-object 
		functors, i.e., the \emph{full embeddings}.
	\end{defi}
	
	In other words, an object of $\F$ is a full embedding
	\begin{equation*}
		A_\tau \xhookrightarrow{j_A} A_\lambda,
	\end{equation*}
	a morphism $f$ from $j_A$ to $j_B$ in $\F$ is given by two functors 
	$f_\tau \colon A_\tau \to B_\tau$ and 
	$f_\lambda \colon A_\lambda \to B_\lambda$ 
	making the following square commute
	\begin{center}
		\begin{tikzcd}
			A_\tau \ar[r, hook, "j_A"] \ar[d, "f_\tau"'] & A_\lambda \ar[d, 
			"f_\lambda"]
			\\
			B_\tau \ar[r, hook, "j_B"'] & B_\lambda
		\end{tikzcd}.
	\end{center}
	
	We call $A_\tau$ the \emph{tight} part of $A$, and $A_\lambda$ the 
	\emph{loose} part of 
	$A$; similarly, we apply this terminology to $f$.
	
	\begin{rk}
		$\F$ is (co)complete and Cartesian closed.
	\end{rk}
	
	An \emph{$\F$-category} $\bbA$ is a category $\bbA$ enriched in $\F$. This 
	means $\bbA$ has
	\begin{enumerate}
		\item [$\bullet$] objects $x, y, \cdots$;
		\item [$\bullet$] hom-objects $\bbA(x, y)$ in $\F$, each consists 
		of a full embedding $\bbA(x, y)_\tau \hookrightarrow \bbA(x, 
		y)_\lambda$ of the tight part into the loose part.
	\end{enumerate}
	
	We can form a $2$-category $\scrA_\tau$ as follows:
	
	\begin{enumerate}
		\item [$\bullet$] $\scrA_\tau$ has all the objects of $\bbA$;
		\item [$\bullet$] the hom-categories $\scrA_\tau (x, y)$ for any 
		objects $x, y$ are $\bbA(x, y)_\tau$.
	\end{enumerate}
	
	Similarly, we can form a $2$-category $\scrA_\lambda$ by setting the 
	hom-categories $\scrA_\lambda (x, y)$  as $\bbA(x, y)_\lambda$ for any 
	objects $x, y$.
	
	Since for each pair of objects $x, y$, $\scrA_\tau (x, y) \hookrightarrow 
	\scrA_\lambda (x, y)$ is a full-embedding, we obtain a $2$-functor
	
	\begin{align*}
		J_\bbA \colon \scrA_\tau &\to \scrA_\lambda.
	\end{align*}
	
	By construction, $J_\bbA$ is identity-on-object, faithful, and locally 
	fully faithful.
	
	\begin{rk}
		We may identify an $\F$-category $\bbA$ with $J_\bbA$. Indeed, any 
		$2$-functor which is identity-on-object, faithful, and locally 
		fully faithful uniquely determined an $\F$-category.
	\end{rk}
	
	The morphisms in $\scrA_\tau$ are called the \emph{tight} morphisms, 
	whereas 
	those in $\scrA_\lambda$ are called \emph{loose}.
	
	\begin{nota}
		We write $A \to B$ for a tight morphism from $A$ to $B$, and $A 
		\leadsto B$ for a loose morphism from $A$ to $B$.
	\end{nota}
	
	\begin{example}[$2$-categories]
		A $2$-category can be viewed as a \emph{chordate} $\F$-category, 
		namely, all 
		the morphisms are assumed to be tight. On the other hand, a 
		$2$-category can also  be viewed as an \emph{inchordate} $\F$-category, 
		namely, only the identities are tight.
	\end{example}
	
	\begin{example}[$\mathbbm{F}$]
		\label{eg:F}
		Since any monoidal closed category is self-enriched, $\F$ is also 
		self-enriched, and we denote this $\F$-category by $\bbF$.
		
		More precisely, $\bbF$ consists of
		\begin{enumerate}
			\item[$\bullet$] objects the full embeddings $A_\tau 
			\hookrightarrow 
			A_\lambda$;
			\item[$\bullet$] tight morphisms the tightness-preserving functors, 
			depicted as
			\begin{center}
				\begin{tikzcd}
					A_\tau \ar[r, hook, "j_A"] \ar[d, "f_\tau"'] & A_\lambda 
					\ar[d, 
					"f_\lambda"]
					\\
					B_\tau \ar[r, hook, "j_B"'] & B_\lambda 
				\end{tikzcd};
			\end{center}
			\item[$\bullet$] loose morphisms the functors $A_\lambda \to 
			B_\lambda$;
			\item[$\bullet$] $2$-cells the natural transformations between the 
			loose morphisms with tight components, or equivalently, the pair of 
			two natural 
			transformations 
			$\alpha_\tau$ and $\alpha_\lambda$ making the following diagram 
			commute
			\begin{center}
				\begin{tikzcd}
					A_\tau \ar[r, hook, "j_A"] \ar[d, "g_\tau"', bend right] 
					\ar[d, 
					"f_\tau", 
					bend left] \ar[d, near end, phantom, 
					"{\substack{\Leftarrow \\ \alpha_\tau}}"] & A_\lambda 
					\ar[d, 
					"f_\lambda", bend left] \ar[d, "g_\lambda"', bend right] 
					\ar[d, 
					near 
					end, phantom, 
					"{\substack{\Leftarrow \\ \alpha_\lambda}}"]
					\\
					B_\tau \ar[r, hook, "j_B"'] & B_\lambda
				\end{tikzcd},
			\end{center}
			as $j_B$ is fully faithful.
		\end{enumerate}
	\end{example}
	
	\begin{example}[$T$-algebras]
		We can combine $\TAlg_s$and $\TAlg_l$ into an 
		$\F$-category $\FTAlg_{s, l}$. The colax and the pseudo cases 
		are also similar.
	\end{example}
	
	\subsection{\texorpdfstring{$\F$}{F}-functors and 
		\texorpdfstring{$\F$}{F}-natural transformations}
	\label{subsec:F-funs_F-nts}
	
	Let $\bbA$ and $\bbB$ be two $\F$-categories. An \emph{$\F$-functor} 
	$F\colon 
	\bbA \to \bbB$ is a functor enriched in $\F$, which precisely means that 
	$F$ consists 
	of $2$-functors $F_\tau \colon \scrA_\tau \to 
	\scrB_\tau$ and $F_\lambda \colon \scrA_\lambda \to \scrB_\lambda$ 
	making the following diagram commute
	
	\begin{center}
		\begin{tikzcd}
			\scrA_\tau \ar[r, hook, "J_\bbA"] \ar[d, "F_\tau"'] & \scrA_\lambda 
			\ar[d, 
			"F_\lambda"]
			\\
			\scrB_\tau \ar[r, hook, "J_\bbB"'] & \scrB_\lambda
		\end{tikzcd}.
	\end{center}
	
	\begin{rk}
		$F_\tau$ is uniquely determined by $F_\lambda$: an $\F$-functor 
		$F\colon 
		\bbA \to \bbB$ is a $2$-functor $F_\lambda \colon 
		\scrA_\lambda \to \scrB_\lambda$ which preserves tightness, i.e., 
		$F_\lambda$ sends a tight morphism in $\bbA$ to a tight morphism in 
		$\bbB$.
	\end{rk}
	
	When $\bbB = \bbF$, an $\F$-functor $F \colon \bbD \to \bbF$ is called an 
	\emph{$\F$-weight}.	Such an $\F$-weight $\Phi \colon \bbD \to 
	\bbF$ amounts to a triple $(\Phi_\tau, \Phi_\lambda, \varphi)$, where 
	$\Phi_\tau 
	\colon \scrD_\tau \to \Cat$ and $\Phi_\lambda \colon \scrD_\lambda \to 
	\Cat$ are $2$-functors, and $\varphi \colon \Phi_\tau \to \Phi_\lambda 
	J_\bbD$ is a $2$-natural transformation
	\begin{center}
		\begin{tikzcd}
			\scrD_\tau \ar[dr, hook, "J_\bbD"'] \ar[rr, "\Phi_\tau"{name=T}] & & 
			\Cat
			\\
			& \scrD_\lambda \ar[ur, "\Phi_\lambda"'] \ar[Rightarrow, from=T, 
			shorten = 3mm, 
			"\varphi"] &
		\end{tikzcd}
	\end{center}
	such that 
	all the components are full-embeddings.
	
	Let $F, G\colon \bbA \rightrightarrows \bbB$ be two $\F$-functors. An 
	\emph{{$\F$}-natural 
		transformation}	$\alpha \colon F \to G$ consists of $2$-natural 
	transformations $\alpha_\tau \colon F_\tau \to G_\tau$ and 
	$\alpha_\lambda \colon F_\lambda \to G_\lambda$ making the 
	following diagram of $2$-cells commute
	
	\begin{center}
		\begin{tikzcd}
			\scrA_\tau \ar[r, hook, "J_\bbA"] \ar[d, "G_\tau"', bend right] 
			\ar[d, 
			"F_\tau", 
			bend left] \ar[d, near end, phantom, 
			"{\substack{\Leftarrow \\ \alpha_\tau}}"] & \scrA_\lambda \ar[d, 
			"F_\lambda", bend left] \ar[d, "G_\lambda"', bend right] \ar[d, 
			near 
			end, phantom, 
			"{\substack{\Leftarrow \\ \alpha_\lambda}}"]
			\\
			\scrB_\tau \ar[r, hook, "J_\bbB"'] & \scrB_\lambda
		\end{tikzcd}.
	\end{center}
	
	\begin{rk}
		\label{rk:F-nat-tran}
		The existence of $\alpha_\tau$ can be seen as 
		the condition that  $\alpha_\lambda$ has tight components.
	\end{rk}
	
	Let $\Psi = (\Psi_\tau, \Psi_\lambda, \psi)$ be another $\F$-weight. An 
	$\F$-natural transformation $\beta \colon \Phi \to \Psi$ consists of 
	$2$-natural 
	transformations $\beta_\tau \colon \Phi_\tau \to 
	\Psi_\tau$ and 
	$\beta_\lambda \colon \Phi_\lambda \to 
	\Psi_\lambda$ satisfying
	\begin{equation}
		\label{eqt:tight_mor_between_weights}
		\begin{tikzcd}[row sep = huge, column sep = large]
			\scrD_\tau \ar[rr,  "\Phi_\tau"{name=T}] \ar[dr, hook, 
			"J_\bbD"'{name=L}]  & & \Cat
			\\
			& \scrD_\lambda \ar[Rightarrow, from=T, shorten=6mm, "\varphi"] 
			\ar[ur, bend 
			left=15, "{\Phi_\lambda}"{name=M}] 
			\ar[ur, bend right=30, "{\Psi_\lambda}"'{name=R}] 
			\ar[Rightarrow, from=M, to=R, shorten=2.5mm, pos=0.3,
			"\scriptstyle{\beta_\lambda}"] & 
		\end{tikzcd}
		=
		\begin{tikzcd}[row sep = huge, column sep = large]
			\scrD_\tau \ar[rr, "\Phi_\tau"{name=T}] \ar[rr, bend right = 30, 
			pos=0.488, "\Psi_\tau"'{name=M}] \ar[dr, hook, "J_\bbD"'] 
			\ar[Rightarrow, 
			from=T, to=M, shorten = 2mm, "\beta_\tau"] & & \Cat
			\\
			& \scrD_\lambda \ar[ur, bend right = 30, "\Psi_\lambda"'] 
			\ar[Rightarrow, from=M, shorten = 0.5mm, "\psi"] &
		\end{tikzcd},
	\end{equation}
	in other words, $\beta_\lambda J_\bbD \varphi = \psi \beta_\tau$.
	
	\subsection{Functor \texorpdfstring{$\F$}{F}-categories}
	\label{subsec:functor_F-cats}
	
	Let $\bbD$ and $\bbK$ be two $\F$-categories, where $\bbD$ is small, i.e., 
	the objects of $\bbD$ and the morphisms in $\bbD$ both form a set, 
	respectively.
	
	We form a \emph{functor $\F$-category} $[\bbD, \bbK]$ as follows:
	
	\begin{enumerate}
		\item [$\bullet$] $[\bbD, \bbK]$ has objects the $\F$-functors $\bbD 
		\to 
		\bbK$;
		\item [$\bullet$] the tight morphisms in $[\bbD, \bbK]_\tau$ are given 
		by 
		the $\F$-natural transformations between the $\F$-functors;
		\item [$\bullet$] the loose morphisms in $[\bbD, \bbK]_\lambda$ are 
		given 
		by the loose part of the $\F$-natural transformations, i.e.,  
		$2$-natural 
		transformations between the loose parts of the 
		$\F$-functors;
		\item [$\bullet$] the $2$-cells are given by the modifications between 
		the 
		loose part of the $\F$-natural transformations. 
	\end{enumerate}

	\section{Dotted $2$-limits and their equivalence to 
		\texorpdfstring{$\F$}{F}-weighted limits}
	\label{sec:dotted}
	
	The notion of marked limits in $2$-category theory gives another 
	perspective to 
	view $2$-categorical limits. Comparing with the notion of $\Cat$-weighted 
	limits, 
	marked 
	limits give a more spread-out presentation, which is easier to track and 
	realise. This motivates our notion of dotted $2$-limits in enhanced 
	$2$-category 
	theory, which functions in the same manner as marked limits. Dotted $2$-limits often
	provide a cleaner and more convenient expression than $\F$-weighted limits. 
	Just as marked limits are designed exclusively for limits in $2$-category 
	theory, dotted $2$-limits are designed exclusively for limits in enhanced 
	$2$-category theory.
	
	\subsection{Dotted $2$-limits}
	\label{subsec:dotted}
	
	\begin{defi}
		\label{def:marked_F-cat}
		Let $\bbD$ be an $\F$-category. Let $\Sigma$ be a class of morphisms in 
		$\bbD$, which contains all the 
		identities and is closed under composition. The pair $(\bbD, \Sigma)$ 
		is called a \emph{marked $\F$-category}.
	\end{defi}

	\begin{nota}
		We denote by $A \nrightarrow B$ a tight morphism from $A$ to $B$ in 
		$\Sigma$, and by $A \amsnleadsto B$ a loose morphism in $\Sigma$.
	\end{nota}
	
	\begin{defi}
		\label{def:dotted_F-cat}
		Let $(\bbD, \Sigma)$ be a marked $\F$-category. Let $T$ be a 
		collection 
		of objects in $\bbD$ such that if $a \in T$ and there is a 
		tight morphism $a \nrightarrow b$ in $\Sigma$, then $b \in 
		T$. The triple $(\bbD, \Sigma, T)$ 
		is called a \emph{dotted $\F$-category}.
	\end{defi}
	
	\begin{nota}
		For an object $t$ in $T$, we highlight it in diagrams with a dot above: 
		$\dot{t}$.
	\end{nota}
	
	\begin{defi}
		\label{def:dottedtrans}
		Let $(\bbD, \Sigma, T)$ be a dotted $\F$-category. Let $S, R 
		\colon \bbD \rightrightarrows \bbA$ be two $\F$-functors. A 
		\emph{dotted-lax natural transformation}  $\alpha \colon S \to 
		R$ 
		between $S$ and $R$ is a marked-lax natural transformation 
		$\alpha_\lambda \colon 
		S_\lambda \to R_\lambda$ such that and for any ${t} 
		\in T$, 
		the $1$-component $\alpha_{{t}}$ is a tight morphism in $\bbA$.
	\end{defi}

	\begin{rk}
		We also have about \emph{dotted-colax} or 
		\emph{dotted-pseudo} 
		natural transformations.
	\end{rk}
	
	\begin{nota}
		We denote the $\F$-category of $\F$-functors $\bbD \to \bbA$, marked-lax 
		natural 
		transformations between loose parts and dotted-lax natural 
		transformations, 
		and modifications by $[\bbD, \bbA]_{l, \Sigma, T}$.
	\end{nota}
	
	\begin{nota}
		We denote by 
		$\Sigma_\lambda(a, b)$ the full subcategory of $\scrD_\lambda(a, b)$ 
		consisting of loose
		morphisms in $\Sigma$
		that have fixed source $a$ and fixed target $b$, and by $\Sigma_\tau(a, 
		b)$ 
		the full subcategory of $\scrD_\tau(a, b)$ consisting of tight
		morphisms in $\Sigma$
		that have fixed source $a$ and fixed target $b$.
		
		It is clear that we have $2$-categories $\Sigma_\lambda$ and 
		$\Sigma_\tau$, 
		which are wide sub-$2$-categories of $\scrD_\lambda$ and $\scrD_\tau$, 
		respectively, and hence there are inclusions $J_{\Sigma_\lambda} \colon 
		\Sigma_\lambda \hookrightarrow \scrD_\lambda$ and $J_{\Sigma_\tau} 
		\colon 
		\Sigma_\tau \hookrightarrow \scrD_\tau$.
	\end{nota}
	
	\begin{defi}
		Let $(\bbD, \Sigma, T)$ be a small dotted $\F$-category, and 
		$\bbA$ be an $\F$-category. The \emph{dotted-lax limit} 
		$\dotlim{l}{S}$ of an $\F$-functor $S 
		\colon \bbD \to 
		\bbA$ is 
		characterised by an isomorphism
		\begin{equation}
			\label{def:dotlim}
			\bbA(A, \dotlim{l}{S}) \cong [\bbD, \bbA]_{l, 
				\Sigma, T}(\triangle(A), 
			S)
		\end{equation}
		in $\bbF$, which is natural in $A \in \ob\bbA$.
	\end{defi}
	
	\begin{rk}
		By replacing dotted-lax natural transformations with dotted-colax or 
		dotted-pseudo natural transformations, we obtain the notions of 
		dotted-colax or dotted-pseudo limits, respectively.
	\end{rk}
	
	\subsection{Examples of dotted $2$-limits}
	\label{subsec:eg}
	
	To illustrate the convenience and practicality of dotted $2$-limits, we 
	describe several examples of important $\F$-weighted limits in the form of 
	dotted $2$-limits. Many of these $\F$-weighted limits have several variations: 
	the $l$-rigged, $c$-rigged, and $p$-rigged versions, which were first 
	discussed in \longcite{LS:2012}.
	
	In the forthcoming examples, we see that the presentation of the 
	$\F$-categorical limits in dotted $2$-limits are simpler than in $\F$-weighted 
	limits. In fact, the indexing dotted $\F$-categories that we give are 
	almost the 
	same as 
	the indexing $\F$-categories given in \longcite{LS:2012}, except that we 
	are marking some morphisms and 'dotting' some objects. In other words, the 
	$\F$-weights are eliminated completely. The data of the dotted 
	$\F$-categories alone already suffice to capture the 
	$\F$-categorical limits of interests.

	In each of the following examples, we denote by $L$ the corresponding 
	dotted $2$-limit of an $\F$-functor $S \colon \bbD \to \bbA$. Besides, we 
	provide the description of $p$-rigged and $c$-rigged limits in terms of 
	dotted-lax limits, whereas that of $l$-rigged limits in terms of 
	dotted-colax limits, as this formulation facilitates our later discussion 
	on the lifting theorem.
	
	We first investigate the examples such that in their indexing 
	dotted $\F$-categories, $\Sigma$ coincides with $\scrD_\tau$.

	\begin{example}[$w$-rigged inserters]
		\label{eg:l-rigged_inserters}
		
		Recall in \longcite{LS:2012} that an $l$-rigged inserter can 
		be formed with the indexing $\F$-category $$\bbB = \{
		\begin{tikzcd}
			x \ar[r, 
			shift 
			left=1ex, "g"] \ar[r, loose, shift right=1ex, "f"']  & 
			y  
		\end{tikzcd}
		\},$$ where the image under the loose part $\Phi_\lambda$ of the weight 
		$\Phi$ is 
		$$\begin{tikzcd}
			\mathbf{1} \ar[r, 
			shift 
			left=1ex, "0"] \ar[r, shift right=1ex, "1"']  & 
			\mathbf{2}  
		\end{tikzcd},$$ 
		and that under the tight part $\Phi_\tau$ is the identity
		$$\begin{tikzcd}
			\mathbf{1} \ar[r] & \mathbf{1}
		\end{tikzcd},$$ 
		and where $\varphi$ has components at $x$ 
		and $y$ being the 
		identity and $0$, respectively. Therefore, the object $0$ in 
		$\mathbf{2}$ is tight.
		
		Let $$\bbD = \{
		\begin{tikzcd}
			\dot{x} \ar[r, 
			"{\scalebox{0.4}[0.7]{\contour{black}{/}}}"{anchor=center,sloped}, 
			shift 
			left=1ex, "g"] \ar[r, loose, shift right=1ex, "f"']  & 
			\dot{y}  
		\end{tikzcd}
		\},$$ which is indeed the same as $\bbB$, by noticing that 
		$\Sigma$ coincides with $\scrD_\tau$, except that we now set 
		the objects $x$ and $y$ to be dotted.
		
		A dotted-colax natural transformation 
		$\alpha \colon 
		\triangle(L) \to S_\lambda$ clearly corresponds to
		\begin{center}
			\begin{tikzcd}[row sep = tiny, column sep = tiny]
				& Sx \ar[dr, 
				shorten <= -0.2cm, shorten >= -0.2cm,
				"Sg"] 
				\ar[dd, Rightarrow, shorten=3mm, "\alpha_f"]&
				\\
				L \ar[ur, shorten <= -0.2cm, shorten >= -0.2cm, 
				"\alpha_x"] \ar[dr, shorten <= -0.2cm, shorten >= -0.2cm, 
				"\alpha_x"'] & & Sy
				\\
				& Sx \ar[ur, loose, start anchor={[xshift=-0.15cm, 
					yshift=-0.15cm]}, end 
				anchor={[xshift=0.2cm, yshift=0.15cm]}, "Sf"'] 
				&
			\end{tikzcd},
		\end{center}
		an $l$-rigged inserter, as desired.

		Similarly, a $c$-rigged inserter can be described using 
		dotted $2$-limits with the same $\bbD$ for $l$-rigged inserters as above, 
		except that we  
		replace dotted-colax natural transformations with dotted-lax 
		natural transformations. 
		
		Recall in \longcite{LS:2012} that a $p$-rigged inserter can 
		be formed with the indexing $\F$-category $$\bbB = \{
		\begin{tikzcd}
			x \ar[r, loose, 
			shift 
			left=1ex, "g"] \ar[r, loose, shift right=1ex, "f"']  & 
			y  
		\end{tikzcd}
		\},$$ where the image under the loose part $\Phi_\lambda$ of the weight 
		$\Phi$ is 
		$$\begin{tikzcd}
			\mathbf{1} \ar[r, 
			shift 
			left=1ex, "0"] \ar[r, shift right=1ex, "1"']  & 
			\mathbf{2}  
		\end{tikzcd},$$ 
		and that $\Phi_\tau(x) = \mathbf{1}$, and 
		$\Phi_\tau(y) = \emptyset$,
		and where $\varphi$ has components at $x$ 
		and $y$ being the 
		identity and the unique morphism $\emptyset \to \mathbf{2}$, 
		respectively. This implies that no objects in $\mathbf{2}$ are tight.
		
		Now, consider the dotted $\F$-category $$\bbD = \{
		\begin{tikzcd}
			\dot{x} \ar[r, 
			"{\scalebox{0.4}[0.7]{\contour{black}{/}}}"{anchor=center,sloped},
			loose, shift 
			left=1ex, "g"] \ar[r, loose, shift right=1ex, "f"']  & 
			{y}  
		\end{tikzcd}
		\},$$ which is very much the same as $\bbB$, except that $g \in \Sigma$ 
		and 
		$x \in T$. 
		
		A dotted-lax natural transformation 
		$\alpha \colon 
		\triangle(L) \to S_\lambda$ clearly corresponds to
		\begin{center}
			\begin{tikzcd}[row sep = tiny, column sep = tiny]
				& Sx \ar[dr, 
				loose, start anchor={[xshift=-0.15cm, 
					yshift=0.15cm]}, end 
				anchor={[xshift=0.2cm, yshift=-0.15cm]},
				"Sf"] 
				\ar[dd, Rightarrow, shorten=3mm, "\alpha_f"]&
				\\
				L \ar[ur, shorten <= -0.2cm, shorten >= -0.2cm, 
				"\alpha_x"] \ar[dr, shorten <= -0.2cm, shorten >= 
				-0.2cm, 
				"\alpha_x"'] & & Sy
				\\
				& Sx \ar[ur, loose, start anchor={[xshift=-0.15cm, 
					yshift=-0.15cm]}, end 
				anchor={[xshift=0.2cm, yshift=0.15cm]}, "Sg"'] 
				&
			\end{tikzcd},
		\end{center}
		giving a $p$-rigged inserter.
	\end{example}

	\begin{example}[$l$-rigged $l$-descent and $c$-rigged $c$-descent objects]
		\label{eg:l-rigged_l-des}
		%		Consider the functors
		%		\begin{center}
			%			\begin{tikzcd}
				%				\dot{\mathbf{1}} \ar[r, 
				%				
				%"{\scalebox{0.4}[0.7]{\contour{black}{/}}}"{anchor=center,sloped},
				%				shift 
				%				left=2ex, "\delta_0"] \ar[r,  loose, 
				%				shift 
				%				right=2ex, "\delta_1"'] & 
				%				\dot{\mathbf{2}} \ar[l, 
				%				
				%"{\scalebox{0.4}[0.7]{\contour{black}{/}}}"{anchor=center,sloped},
				%					shift 
				%				right=0.5ex, 
				%				"\sigma"] \ar[r, 
				%				
				%"{\scalebox{0.4}[0.7]{\contour{black}{/}}}"{anchor=center,sloped},
				%				shift 
				%				left=0.5ex,
				%				"\delta_1"']
				%				\ar[r, 
				%				
				%"{\scalebox{0.4}[0.7]{\contour{black}{/}}}"{anchor=center,sloped},
				%				shift 
				%				left=2ex, "\delta_0"] \ar[r,  loose, 
				%				shift 
				%				right=2ex, "\delta_2"'] &
				%				\dot{\mathbf{3}}
				%			\end{tikzcd},
			%		\end{center}
		%		where $\delta_i$ is the inclusion omitting $i$ for $i = 0, 1, 
		%2$, and 
		%		$\sigma$ is the 
		%		constant map at $0$ in $\mathbf{1}$.
		We start with the $l$-rigged $l$-descent objects.
		
		Let $\bbD$ be the locally discrete sub-$\F$-category such that
		$\scrD_\lambda$ is generated by
		\begin{center}
			\begin{tikzcd}
				\dot{\mathbf{1}} \ar[r, 
				"{\scalebox{0.4}[0.7]{\contour{black}{/}}}"{anchor=center,sloped},
				shift 
				left=2ex, "\delta_0"] \ar[r,  loose, 
				shift 
				right=2ex, "\delta_1"'] & 
				\dot{\mathbf{2}} \ar[l, 
				"{\scalebox{0.4}[0.7]{\contour{black}{/}}}"{anchor=center,sloped},
				shift 
				right=0.5ex, 
				"\sigma"] \ar[r, 
				"{\scalebox{0.4}[0.7]{\contour{black}{/}}}"{anchor=center,sloped},
				shift 
				left=0.5ex,
				"\delta_1"']
				\ar[r, 
				"{\scalebox{0.4}[0.7]{\contour{black}{/}}}"{anchor=center,sloped},
				shift 
				left=2ex, "\delta_0"] \ar[r,  loose, 
				shift 
				right=2ex, "\delta_2"'] &
				\dot{\mathbf{3}}
			\end{tikzcd},
		\end{center}
		whereas $\scrD_\tau$ is generated by
		\begin{center}
			\begin{tikzcd}
				\dot{\mathbf{1}} \ar[r, 
				"{\scalebox{0.4}[0.7]{\contour{black}{/}}}"{anchor=center,sloped},
				shift 
				left=1ex, "\delta_0"]  & 
				\dot{\mathbf{2}} \ar[l, 
				"{\scalebox{0.4}[0.7]{\contour{black}{/}}}"{anchor=center,sloped},
				shift 
				left=1ex, 
				"\sigma"] \ar[r, 
				"{\scalebox{0.4}[0.7]{\contour{black}{/}}}"{anchor=center,sloped},
				shift 
				right=1ex,
				"\delta_1"']
				\ar[r, 
				"{\scalebox{0.4}[0.7]{\contour{black}{/}}}"{anchor=center,sloped},
				shift 
				left=1ex, "\delta_0"]  &
				\dot{\mathbf{3}}
			\end{tikzcd}.
		\end{center}
		This dotted $\F$-category is indeed the same as the indexing 
		$\F$-category provided in \longcite{LS:2012}, except that we 
		have dotted objects. Note that the corresponding $\F$-weight given 
		there is non-trivial.
		
		Let the image of an $\F$-functor $S \colon \bbD \to \bbA$ takes the form
		\begin{center}
			\begin{tikzcd}
				A_0 \ar[r, 
				shift 
				left=2ex, "\delta_0^A"] \ar[r,  loose, 
				shift 
				right=2ex, "\delta_1^A"'] & 
				A_1 \ar[l, 
				"{\sigma}^A" description] \ar[r, 
				"\delta_1^A"' description]
				\ar[r, 
				shift 
				left=2ex, "\delta_0^A"] \ar[r,  loose, 
				shift 
				right=2ex, "\delta_2^A"'] &
				A_2
			\end{tikzcd},
		\end{center}
		where $A_{i} := S\mathbf{i + 1}$ and $\delta_i^A := S\delta_i$ for $i 
		= 0, 1, 2$, and $\sigma^A := S\sigma$.
		Then a dotted-colax natural transformation $\alpha \colon \triangle(L) 
		\to S_\lambda$ actually gives
		\begin{center}
			\begin{tikzcd}[row sep = tiny, column sep = tiny]
				& A_0 \ar[dr, shorten <= -0.2cm, shorten >= -0.2cm,
				"\delta_0^A"] 
				\ar[dd, Rightarrow, shorten=3mm, "\alpha_{\delta_1}"]&
				\\
				L \ar[ur, shorten <= -0.2cm, shorten >= -0.2cm, 
				"a_0"] \ar[dr, shorten <= -0.2cm, shorten >= -0.2cm, 
				"a_0"'] & & A_1
				\\
				& A_0 \ar[ur,  loose, start anchor={[xshift=-0.15cm, 
					yshift=-0.15cm]}, end 
				anchor={[xshift=0.2cm, yshift=0.15cm]}, "\delta_1^A"'] 
				&
			\end{tikzcd},
		\end{center}
		that satisfies the equations for $l$-descent objects,
		%		\begin{equation*}
			%			\delta_1^A \cdot \alpha_{\delta_1} = (\delta_2^A \cdot 
			%			\alpha_{\delta_1}) \circ (\delta_0^A \cdot 
			%\alpha_{\delta_1}) ,  
			%			\text{and} 
			%			\hspace{0.5em} \sigma^A \cdot \alpha_{\delta_1} = 1,
			%		\end{equation*}
		%		similar to the case of marked limits, giving the $l$-descent 
		%object in 
		%		$\scrA_\lambda$, 
		where all the projections $A \xrightarrow{a_i} A_i$ 
		are	tight because $\mathbf{1}$, $\mathbf{2}$ and $\mathbf{3}$ are 
		all in $T$.
		
		Since 
		$\delta_0^A$'s are already tight, this means that the 
		projection
		$a_0$ is tight and detects tightness.
		
		Now for $c$-rigged $c$-descent objects, we again adapt 
		the same dotted $\F$-category $\bbD$ as above, and consider
		dotted-lax natural transformations instead.
	\end{example}
	
	\begin{example}[$w$-rigged equifiers]
		\label{eg:l-rigged_equifiers}
		We start with the $l$-rigged equifiers.
		
		Let $$\bbD = \{
		\begin{tikzcd}
			\dot{x} \ar[r, 
			"{\scalebox{0.4}[0.7]{\contour{black}{/}}}"{anchor=center,sloped}, 
			start 
			anchor={[xshift=-0.15cm, 
				yshift=-0.15cm]}, end 
			anchor={[xshift=0.15cm, yshift=-0.15cm]}, shift 
			left=1ex, bend left, "f"{name=T}] \ar[r, 
			loose,  
			start anchor={[xshift=-0.15cm, 
				yshift=0.15cm]}, end 
			anchor={[xshift=0.15cm, yshift=0.15cm]},
			shift right=1ex, bend right, "g"'{name=B}]  & 
			\dot{y}  \arrow[Rightarrow, shift 
			left=1ex, shorten=2mm, from=T, to=B, "\beta"] \arrow[Rightarrow, 
			shift 
			right=1ex, shorten=2mm, from=T, to=B, "\alpha"']
		\end{tikzcd}
		\},$$ which again is the same as the $\F$-category used to 
		describe $l$-rigged equifiers by $\F$-weighted limits, 
		except that we have dotted objects.
		
		Let $\gamma \colon \triangle(L) \to S_\lambda$ be a dotted-colax natural transformation.
		
		%		A dotted-colax natural transformation $\gamma \colon 
		%		\triangle(L) \to S_\lambda$ has data
		%		\begin{center}
			%			\begin{tikzcd}[row sep = normal, column sep = small]
				%				& L \ar[dl, "\gamma_x"'] \ar[dr, 
				%				"\gamma_y"{name=R}]  &
				%				\\
				%				Sx \ar[rr, loose, "Sg"'] \ar[Rightarrow, 
				%shorten=5mm, 
				%				from=R, 
				%				"\gamma_g"]   & & Sy
				%			\end{tikzcd}
			%			\begin{tikzcd}[row sep = normal, column sep = small]
				%				& L \ar[dl, "\gamma_x"'] \ar[dr, 
				%				"\gamma_y"{name=R}]  &
				%				\\
				%				Sx \ar[rr,  "Sf"']   & & Sy
				%			\end{tikzcd},
			%		\end{center}
		By the naturality of $\gamma$, we have
		\begin{equation*}
			\begin{tikzcd}[row sep = normal, column sep = small]
				& L \ar[dl, "\gamma_x"'] \ar[dr, 
				"\gamma_y"{name=R}]  &
				\\
				Sx \ar[rr, start anchor={[xshift=-0.15cm, 
					yshift=-0.15cm]}, end 
				anchor={[xshift=0.15cm, yshift=-0.15cm]},  bend left, 
				"Sf"{name=T}] \ar[rr, loose, start 
				anchor={[xshift=-0.15cm, 
					yshift=0.15cm]}, end 
				anchor={[xshift=0.15cm, yshift=0.15cm]},
				bend right, "Sg"'{name=B}] \ar[Rightarrow, 
				shorten=2mm, 
				"S\alpha", from=T, to=B]  & & Sy
			\end{tikzcd}
			=
			\begin{tikzcd}[row sep = normal, column sep = small]
				& L \ar[dl, "\gamma_x"'] \ar[dr, 
				"\gamma_y"{name=R}]  &
				\\
				Sx \ar[rr, loose, bend right, "Sg"']  \arrow[Rightarrow,  
				shorten=5mm, 
				from=R, 
				"\gamma_g"] & & Sy
			\end{tikzcd}
			=
			\begin{tikzcd}[row sep = normal, column sep = small]
				& L \ar[dl, "\gamma_x"'] \ar[dr, 
				"\gamma_y"{name=R}]  &
				\\
				Sx \ar[rr, start anchor={[xshift=-0.15cm, 
					yshift=-0.15cm]}, end 
				anchor={[xshift=0.15cm, yshift=-0.15cm]},  bend left, 
				"Sf"{name=T}] \ar[rr, loose, start 
				anchor={[xshift=-0.15cm, 
					yshift=0.15cm]}, end 
				anchor={[xshift=0.15cm, yshift=0.15cm]},
				bend right, "Sg"'{name=B}] \ar[Rightarrow, 
				shorten=2mm, 
				"S\beta", from=T, to=B]  & & Sy
			\end{tikzcd}.
		\end{equation*}
		So the above data assemble to 
		the equifier in $\scrA_\lambda$:
		\begin{center}
			\begin{tikzcd}
				L \ar[r, "\gamma_x"] & {Sx} \ar[r, start 
				anchor={[xshift=-0.15cm, 
					yshift=-0.15cm]}, end 
				anchor={[xshift=0.15cm, yshift=-0.15cm]}, shift 
				left=1ex, bend left, "Sf"{name=T}] \ar[r, loose,  
				start anchor={[xshift=-0.15cm, 
					yshift=0.15cm]}, end 
				anchor={[xshift=0.15cm, yshift=0.15cm]},
				shift right=1ex, bend right, "Sg"'{name=B}]   & 
				{Sy}  \arrow[Rightarrow, shift 
				left=1ex, shorten=2.5mm, from=T, to=B, "\scriptstyle{S\beta}"] 
				\arrow[Rightarrow, 
				shift 
				right=1ex, shorten=2.5mm, from=T, to=B, 
				"\scriptstyle{S\alpha}"']
			\end{tikzcd}.
		\end{center}
		In addition, the projections $\gamma_x$ and $\gamma_y$ are tight, and 
		$\gamma_x$ detects tightness.
		
		%		For a loose morphism $h \colon A \leadsto L$ in 
		%$\scrA_\lambda$, $h$ 
		%		corresponds 
		%		to a marked-colax natural transformation $\triangle(A) 
		%		\xrightarrow{\triangle(h)} \triangle(L) \xrightarrow{\gamma} 
		%		S_\lambda$, 
		%		which is dotted precisely when $A 
		%		\xleadsto{h} L \xrightarrow{\gamma_x} Sx$ and  $A \xleadsto{h} 
		%L 
		%		\xrightarrow{\gamma_x} Sx \xrightarrow{Sf} Sy$ are tight. Since 
		%$Sf$ is 
		%		automatically tight, in other words, the projection
		%		$\gamma_x$ is tight and detects tightness.
		For $c$-rigged equifiers, consider this time when $$\bbD = \{
		\begin{tikzcd}
			\dot{x} \ar[r, 
			loose, start 
			anchor={[xshift=-0.15cm, 
				yshift=-0.15cm]}, end 
			anchor={[xshift=0.15cm, yshift=-0.15cm]}, shift 
			left=1ex, bend left, "f"{name=T}] \ar[r, 
			"{\scalebox{0.4}[0.7]{\contour{black}{/}}}"{anchor=center,sloped},
			start anchor={[xshift=-0.15cm, 
				yshift=0.15cm]}, end 
			anchor={[xshift=0.15cm, yshift=0.15cm]},
			shift right=1ex, bend right, "g"'{name=B}]  & 
			\dot{y}  \arrow[Rightarrow, shift 
			left=1ex, shorten=2mm, from=T, to=B, "\beta"] 
			\arrow[Rightarrow, 
			shift 
			right=1ex, shorten=2mm, from=T, to=B, "\alpha"']
		\end{tikzcd}
		\}.$$ This is also the same as the $\F$-category used in 
		describing the limit as an $\F$-weighted limit.
		
		Now, a dotted-lax natural transformation  $\gamma \colon 
		\triangle(L) \to S_\lambda$ gives our desired $c$-rigged equifiers.
		
		For $p$-rigged equifiers, let $$\bbD = \{
		\begin{tikzcd}
			\dot{x} \ar[r, loose, start 
			anchor={[xshift=-0.15cm, 
				yshift=-0.15cm]}, end 
			anchor={[xshift=0.15cm, yshift=-0.15cm]}, shift 
			left=1ex, bend left, "f"{name=T}] \ar[r, 
			"{\scalebox{0.4}[0.7]{\contour{black}{/}}}"{anchor=center,sloped}, 
			loose,  
			start anchor={[xshift=-0.15cm, 
				yshift=0.15cm]}, end 
			anchor={[xshift=0.15cm, yshift=0.15cm]},
			shift right=1ex, bend right, "g"'{name=B}]  & 
			{y}  \arrow[Rightarrow, shift 
			left=1ex, shorten=2mm, from=T, to=B, "\beta"] \arrow[Rightarrow, 
			shift 
			right=1ex, shorten=2mm, from=T, to=B, "\alpha"']
		\end{tikzcd}
		\}.$$Again, the only difference between the $\F$-category given in 
		\longcite{LS:2012} and this dotted $\F$-category is that $g \in \Sigma$.
		
		A dotted lax-natural transformation  $\gamma \colon 
		\triangle(L) \to S_\lambda$ gives 
		the equifier of $S_\lambda$ in $\scrA_\lambda$:
		\begin{center}
			\begin{tikzcd}
				L \ar[r, "\gamma_x"] & {Sx} \ar[r, loose, start 
				anchor={[xshift=-0.15cm, 
					yshift=-0.15cm]}, end 
				anchor={[xshift=0.15cm, yshift=-0.15cm]}, shift 
				left=1ex, bend left, "Sf"{name=T}] \ar[r, loose,  
				start anchor={[xshift=-0.15cm, 
					yshift=0.15cm]}, end 
				anchor={[xshift=0.15cm, yshift=0.15cm]},
				shift right=1ex, bend right, "Sg"'{name=B}]   & 
				{Sy}  \arrow[Rightarrow, shift 
				left=1ex, shorten=2.5mm, from=T, to=B, "\scriptstyle{S\beta}"] 
				\arrow[Rightarrow, 
				shift 
				right=1ex, shorten=2.5mm, from=T, to=B, 
				"\scriptstyle{S\alpha}"']
			\end{tikzcd}.
		\end{center}
		The	projection $\gamma_x$ is tight and detects tightness.
	\end{example}

	\begin{example}[Alternating projective limits]
		\label{eg:alt}
		This is an example of a limit which is not PIE but lifts to the 
		$\F$-category of algebras, as shown in \cite[Example 6.8]{LS:2012}.
		
		Consider the opposite poset of natural numbers 
		\begin{center}
			\begin{tikzcd}
				\dot{1} & 2 \ar[l, 
				"{\scalebox{0.4}[0.7]{\contour{black}{/}}}"{anchor=center,sloped}]
				
				& 
				\dot{3} \ar[l, loose, 
				"{\scalebox{0.4}[0.7]{\contour{black}{/}}}"{anchor=center,sloped}]
				
				& 4 
				\ar[l, 
				"{\scalebox{0.4}[0.7]{\contour{black}{/}}}"{anchor=center,sloped}]
				
				\cdots
			\end{tikzcd},
		\end{center}
		and denote this $\F$-category by $\bbD$, where all the odd numbers are 
		in 
		$T$, all the morphisms are in $\Sigma$, and the unique morphism $n 
		\amsnleadsto m$ is tight precisely when it is the identity or when $n$ 
		is 
		even 
		but $m$ is odd.
		
		A dotted-lax natural transformation $\alpha \colon \triangle(L) 
		\to S_\lambda$ gives
		%		\begin{center}
			%			\begin{tikzcd}[row sep = normal, column sep = small]
				%				& L \ar[dl, loose, "{\alpha_{2k}}"'] \ar[dr,  
				%"{\alpha_{2k - 
						%				1}}"] &
				%				\\
				%				S(2k) \ar[rr, "Sg"'] & & S(2k - 1)
				%			\end{tikzcd}
			%			\begin{tikzcd}[row sep = normal, column sep = small]
				%				& L \ar[dl, "\alpha_{2k + 1}"'] \ar[dr, loose, 
				%"\alpha_{2k}"] &
				%				\\
				%				S(2k + 1) \ar[rr, "Sg"'] & & S(2k)
				%			\end{tikzcd}
			%		\end{center}
		%		for each positive integer $k$. These commutative triangles 
		%combine to 
		%		give
		\begin{center}
			\begin{tikzcd}
				& & L \ar[dl, loose] \ar[d] \ar[dr, loose] \ar[drr]& &
				\\
				{\cdots} \ar[r, loose]	& {S(4)}  \ar[r]& {S(3)} \ar[r, loose] 
				& {S(2)} 
				\ar[r] & {S(1)}
			\end{tikzcd},
		\end{center}
		which is the projective limit of $S_\lambda$ in $\scrA_\lambda$, in 
		addition, the projections at odd numbers are tight.
		
		%		Any marked-lax natural transformation $\triangle(A) 
		%		\xrightarrow{\triangle(h)} \triangle(L) \xrightarrow{\alpha} 
		%		S_\lambda$ becomes dotted if and only if $A 
		%		\xleadsto{h} L \xrightarrow{\alpha_{2k - 1}} S(2k - 1)$ is 
		%tight for 
		%		any 
		%		positive integer $k$. This is equivalent to that the 
		%projections at odd 
		%		numbers are tight and detects tightness.
		The projections at odd 
		numbers are tight and detects tightness.
	\end{example}
	
	We will discuss more on the above examples in 
	\longref{Section}{sec:discuss}.
	
	\subsection{The equivalence between \texorpdfstring{$\F$}{F}-limits and 
		dotted 
		limits}
	\label{subsec:F-equiv}
	
	Motivated by the $2$-categorical case, we expect that dotted $2$-limits are as 
	expressive 
	as 
	$\F$-limits. Our main goal in this section is to show that the two notions 
	are equivalent.
	
	\begin{nota}
		In this section, we denote a surjective-on-objects functor by 
		$\twoheadrightarrow$, and a full embedding by $\rightarrowtail$.
	\end{nota}
	
	\begin{rk}
		\label{rk:fs}
		Indeed, surjective-on-objects functors and full embeddings form an 
		orthogonal factorisation system on $\Cat$: 
		
		For any arbitrary functor $F \colon A \to B$, the \emph{embedded image} 
		$\widetilde{\im}F$ of $F$ is the full subcategory of $B$ whose objects 
		are in the image of $F$.
		
		Now, it is clear that we have a factorisation 
		\begin{center}
			\begin{tikzcd}
				A \ar[rr, "F"] \ar[dr, two heads, "E"'] & & 
				B
				\\
				& \widetilde{\im}F \ar[ur, tail, "M"'] &
			\end{tikzcd}
		\end{center}
		of $F$ through $\widetilde{\im}F$, where by default $E$ is a 
		surjective-on-objects functor, and $M$ is a full embedding.
		
		The uniqueness of lifts of surjective-on-objects functors against full 
		embeddings is straightforward.
	\end{rk}
	
	We first show that any dotted-lax limit has the same universal property as 
	an 
	$\F$-weighted limit.
	
	Similar to the $2$-categorical case in section \longref{}{subsec:2-equiv}, 
	we aim 
	to 
	construct a left adjoint 
	$()^\# \colon [\bbD, \bbF]_{l, \Sigma, 
		T} 
	\to [\bbD, \bbF]$ to the inclusion $\F$-functor, whose universal property 
	is given by
	\begin{equation}
		\label{eqt:F_adj_iso}
		[\bbD, \bbF]_{l, 
			\Sigma, T}(F, G) \cong [\bbD, 
		\bbF](F^\#, G).
	\end{equation}
	With this left adjoint $()^\#$, we can then deduce that for any 
	$\F$-functor $S \colon 
	\bbD \to \bbA$ and an object $A \in \bbA$,
	\begin{align*}
		[\bbD, \bbA]_{l, \Sigma, T}(\triangle(A), S-) &\cong 
		[\bbD, \bbF]_{l, 
			\Sigma, T}(\triangle(\mathbf{1}), \bbA(A, S-)) 
		\\
		&\cong [\bbD, 
		\bbF]({\triangle(\mathbf{1})}^\#, \bbA(A, S-)),
	\end{align*}
	as desired.	Indeed, from \longref{Proposition}{pro:sigma_left_adj}, we 
	obtain 
	a $2$-categorical left 
	adjoint $()^\#_\lambda := ()^\ddagger \colon [\scrD_\lambda, \Cat]_{l, 
		\Sigma} 
	\to [\scrD_\lambda, \Cat]$, and so we already have
	\begin{equation*}
		[\scrD_\lambda, \Cat]_{l, 
			\Sigma}(F_\lambda, G_\lambda) \cong 
		[\scrD_\lambda, \Cat](F^\#_\lambda, G_\lambda),
	\end{equation*}
	which gives exactly the loose part of our desired isomorphism 
	\eqtref{}{eqt:F_adj_iso}. That means 
	our goal is to construct, for any $\F$-weight $(F_\tau, F_\lambda, 
	\theta)$, a $2$-functor $F^\#_\tau \colon \scrD_\tau \to \Cat$ such that 
	$(F^\#_\tau, F^\#_\lambda, \varphi)$ is an $\F$-weight, and that 
	\eqtref{}{eqt:F_adj_iso} is fulfilled.
	
	\begin{construct}
		\label{constr:tight_part_left_adj}
		Let $\bbD = (\bbD, \Sigma, T)$ be a small dotted $\F$-category. Let $F 
		= (F_\tau, F_\lambda, \theta \colon F_\tau \to 
		F_\lambda J_\bbD)$ be an $\F$-weight.
		
		We view $T$ as a full sub-$2$-category of $\Sigma_\tau$, i.e., $T$ is a 
		$2$-category where every morphism is tight and is in $\Sigma$. Denote 
		by 
		$J^{\scrD_\tau}_T \colon T \hookrightarrow \Sigma_\tau \hookrightarrow 
		\scrD_\tau$ the inclusion of $T$ into $\scrD_\tau$.
		
		Since $\Cat$ is cocomplete and $T$ is small, the left Kan extension $L 
		:= {\Lan[F_\tau 
			J^{\scrD_\tau}_T]{J^{\scrD_\tau}_T}}$ of $F_\tau 
		J^{\scrD_\tau}_T$ along $J^{\scrD_\tau}_T$ 
		\begin{center}
			\begin{tikzcd}
				T \ar[r, hook, "J^{\scrD_\tau}_T"] \ar[dr, hook, 
				"J^{\scrD_\tau}_T"'] 
				& \scrD_\tau \ar[r, "F_\tau"] \ar[Rightarrow, dashed 
				,shorten=1.5mm, 
				to=2-2, "\pi"] & \Cat
				\\
				& \scrD_\tau \ar[ur, dashed, "L"']
			\end{tikzcd}
		\end{center}
		exists. On the other hand, 
		we have $2$-natural transformations $\theta \cdot J^{\scrD_\tau}_T 
		\colon F_\tau J^{\scrD_\tau}_T \to F_\lambda J_\bbD J^{\scrD_\tau}_T$ 
		and $\eta_F \cdot J_\bbD \cdot J^{\scrD_\tau}_T$, where $\eta_F \colon 
		F_\lambda \to F^\#_\lambda$ denotes the component of the unit of the 
		$2$-categorical adjunction in 
		\longref{Proposition}{pro:sigma_left_adj}. 
		Therefore, we have a composite
		\begin{center}
			\begin{tikzcd}[row sep = large, column sep = large]
				T \ar[r, hook, "J^{\scrD_\tau}_T"] \ar[dr, hook, 
				"J^{\scrD_\tau}_T"'{name=L}]  & \scrD_\tau \ar[r, "F_\tau"] 
				\ar[Rightarrow, to=2-2, shorten=3mm, "\scriptstyle{\theta 
					\cdot 
					J^{\scrD_\tau}_T }"'] & \Cat
				\\
				& \scrD_\tau \ar[ur, bend left=15, "\scriptstyle{F_\lambda 
					J_\bbD}"{name=M}] 
				\ar[ur, bend right=30, "\scriptstyle{F^\#_\lambda 
					J_\bbD}"'{name=R}] 
				\ar[Rightarrow, from=M, to=R, shorten=3mm, pos=0.3,
				"\scriptscriptstyle{\eta_F 
					\cdot 
					J_\bbD}"] & 
			\end{tikzcd}
		\end{center}
		of $2$-natural transformations, and by the universal property of the 
		left Kan extension, there exists a unique $2$-natural transformation $l 
		\colon L \to F^\#_\lambda J_\bbD$ such that 
		\begin{equation*}
			\begin{tikzcd}[row sep = large, column sep = large]
				T \ar[r, hook, "J^{\scrD_\tau}_T"] \ar[dr, hook, 
				"J^{\scrD_\tau}_T"'{name=L}]  & \scrD_\tau \ar[r, "F_\tau"] 
				\ar[Rightarrow, to=2-2, shorten=3mm, "\scriptstyle{\theta 
					\cdot 
					J^{\scrD_\tau}_T }"'] & \Cat
				\\
				& \scrD_\tau \ar[ur, bend left=15, "\scriptstyle{F_\lambda 
					J_\bbD}"{name=M}] 
				\ar[ur, bend right=30, "\scriptstyle{F^\#_\lambda 
					J_\bbD}"'{name=R}] 
				\ar[Rightarrow, from=M, to=R, shorten=3mm, pos=0.3,
				"\scriptscriptstyle{\eta_F 
					\cdot 
					J_\bbD}"] & 
			\end{tikzcd}
			=
			\begin{tikzcd}[row sep = large, column sep = large]
				T \ar[r, hook, "J^{\scrD_\tau}_T"] \ar[dr, hook, 
				"J^{\scrD_\tau}_T"'{name=L}]  & \scrD_\tau \ar[r, "F_\tau"] 
				\ar[Rightarrow, to=2-2, shorten=3mm, "{\pi}"'] & 
				\Cat
				\\
				& \scrD_\tau \ar[ur, bend left=15, "{L}"{name=M}] 
				\ar[ur, bend right=30, "{F^\#_\lambda 
					J_\bbD}"'{name=R}] 
				\ar[Rightarrow, from=M, to=R, shorten=3mm, dashed,
				"{\exists!l}"] & 
			\end{tikzcd},
		\end{equation*}
		that is to say, $(l \cdot J^{\scrD_\tau}_T) \circ \pi = 
		(\eta_F 
		\cdot 
		J_\bbD \cdot J^{\scrD_\tau}_T) \circ (\theta \cdot J^{\scrD_\tau}_T)$. 
		As a result, for each object $d$ in $\bbD$, we have a functor $l_d 
		\colon Ld \to F^\#_\lambda d$.
		
		Nevertheless, this is not yet the end of our story, as $l_d$ is not 
		necessarily a full embedding. To resolve the issue, let us consider 
		the embedded image $\widetilde{\im}l_d$. Following 
		\longref{Remark}{rk:fs}, we have the factorisation 
		\begin{center}
			\begin{tikzcd}
				Ld \ar[rr, "l_d"] \ar[dr, two heads, "p_d"'] & & 
				{{F^\#_\lambda} d}
				\\
				& \widetilde{L}d \ar[ur, tail, "\widetilde{l}_d"'] &
			\end{tikzcd}
		\end{center}
		of $l_d$ through $\widetilde{L}d$, where by default $p_d$ is a 
		surjective-on-objects functor, and $\widetilde{l}_d$ is a full 
		embedding. Using the orthogonality of surjective-on-objects functors 
		and full embeddings, $\widetilde{L}$ uniquely extends to a $2$-functor 
		such that $p_d$ and $\widetilde{l}_d$ are the $1$-components of the 
		$2$-natural transformations $p \colon L \to \widetilde{L}$ and 
		$\widetilde{l} \colon \widetilde{L} \to F^\#_\lambda J_\bbD$, 
		respectively.
		
		By setting $F^\#_\tau := \widetilde{L}$ and $\varphi := \widetilde{l}$, we 
		finish our construction of the $\F$-weight $F^\# = (F^\#_\tau, 
		F^\#_\lambda, \varphi)$.
	\end{construct}
	
	Now, let $G = (G_\tau, G_\lambda, \gamma \colon G_\tau \to G_\lambda 
	J_\bbD)$ be another $\F$-weight. Since $J_\bbA$ is locally fully 
	faithful, if we view $T$ as a full 
	sub-$2$-category of $\Sigma_\tau$, then the condition that the 
	$1$-component $\alpha_{{t}}$ is tight for $t \in T$ can be seen as 
	the existence of a $2$-natural transformation $\alpha_\delta \colon 
	F_\tau J^{\scrD_\tau}_T \to G_\tau J^{\scrD_\tau}_T$, satisfying 
	\eqtref{}{eqt:tight_mor_between_weights}: 
	\begin{equation}
		\label{eqt:alpha_delta}
		(\alpha_\lambda \cdot J_\bbD J^{\scrD_\tau}_T) \circ (\theta \cdot 
		J^{\scrD_\tau}_T) = (\gamma \cdot J^{\scrD_\tau}_T) \circ \alpha_\delta.
	\end{equation}
	More precisely, a dotted-lax natural transformation 
	$\alpha \colon F \to G$ in $[\bbD, \bbF]_{l, \Sigma}(F, G)$ consists of
	\begin{enumerate}
		\item[$\bullet$] a lax natural transformation $\alpha_\lambda \colon 
		F_\lambda \to G_\lambda$ between the loose parts;
		\item[$\bullet$] a $2$-natural transformation $\alpha_\lambda \cdot 
		J_{\Sigma_\lambda} \colon F_\lambda J_{\Sigma_\lambda} \to G_\lambda 
		J_{\Sigma_\lambda}$;
		\item[$\bullet$] a $2$-natural transformation $\alpha_\delta \colon 
		F_\tau J^{\scrD_\tau}_T \to G_\tau J^{\scrD_\tau}_T$, 
	\end{enumerate} 
	that fulfill \eqtref{}{eqt:alpha_delta}, i.e.,
	\begin{equation*}
		\begin{tikzcd}[row sep = large, column sep = large]
			T \ar[r, hook, "J^{\scrD_\tau}_T"] \ar[dr, hook, 
			"J^{\scrD_\tau}_T"'{name=L}]  & \scrD_\tau \ar[r, "F_\tau"] 
			\ar[Rightarrow, to=2-2, shorten=3mm, "\scriptstyle{\theta 
				\cdot 
				J^{\scrD_\tau}_T }"'] & \Cat
			\\
			& \scrD_\tau \ar[ur, bend left=15, "\scriptstyle{F_\lambda 
				J_\bbD}"{name=M}] 
			\ar[ur, bend right=30, "\scriptstyle{G_\lambda 
				J_\bbD}"'{name=R}] 
			\ar[Rightarrow, from=M, to=R, shorten=3.5mm, pos=0.3,
			"\scriptscriptstyle{\alpha_\lambda 
				\cdot 
				J_\bbD}"] & 
		\end{tikzcd}
		=
		\begin{tikzcd}[row sep = large, column sep = large]
			T \ar[r, hook, "J^{\scrD_\tau}_T"] \ar[dr, hook, 
			"J^{\scrD_\tau}_T"'{name=L}]  & \scrD_\tau \ar[r, "F_\tau"] 
			\ar[Rightarrow, to=2-2, shorten=3mm, 
			"{\alpha_\delta}"'] & \Cat
			\\
			& \scrD_\tau \ar[ur, bend left=15, "{G_\tau}"{name=M}] 
			\ar[ur, bend right=30, "{G_\lambda 
				J_\bbD}"'{name=R}] 
			\ar[Rightarrow, from=M, to=R, shorten=3mm, 
			"{\gamma}"] & 
		\end{tikzcd},
	\end{equation*}
	here $\alpha_\lambda \cdot J_\bbD J^{\scrD_\tau}_T \colon F_\lambda J_\bbD 
	J^{\scrD_\tau}_T \to G_\lambda J_\bbD J^{\scrD_\tau}_T$ is a $2$-natural 
	transformation.
	
	In other words, our goal is to show that $\alpha \colon F \to G$ 
	corresponds bijectively to an $\F$-natural transformation $\beta \colon 
	F^\# \to G$, or equivalently, a pair of $2$-natural transformations 
	$\beta_\lambda \colon F^\#_\lambda \to G_\lambda$ and $\beta_\tau \colon 
	F^\#_\tau \to G_\tau$, satisfying \eqtref{}{eqt:tight_mor_between_weights}. 
	Since $\beta_\lambda \colon F^\#_\lambda \to G_\lambda$ is obtained 
	immediately from the $2$-categorical adjunction in 
	\longref{Proposition}{pro:sigma_left_adj}, our task is now reduced to 
	finding 
	an appropriate $2$-natural transformation $\beta_\tau \colon F^\#_\tau \to 
	G_\tau$.
	
	To begin with, we check that in the 
	$2$-categorical adjunction, any unit component and the 
	transposes 
	extend from marked-lax natural 
	transformations to dotted-lax natural transformations.
	
	\begin{pro}
		\label{pro:unit_dotted}
		The component $\eta_F \colon F_\lambda \to F^\#_\lambda$ of the 
		unit of the adjunction in \longref{Proposition}{pro:sigma_left_adj} is 
		a 
		dotted-lax natural transformation.
		
		In addition, given an $\F$-natural transformation $\beta \colon F^\# 
		\to G$, its transpose obtained by pre-composing with the unit $\eta_F$ 
		is also a dotted-lax natural transformation.
	\end{pro}
	
	\begin{proof}
		It suffices to show that for any $t \in T$, ${\eta_F}_t$ is a tight 
		morphism in $\bbF$, i.e., a functor preserving tight parts.
		
		From the construction of $l = \varphi \circ p$ in 
		\longref{Construction}{constr:tight_part_left_adj}, we have $((\varphi 
		\circ p) \cdot J^{\scrD_\tau}_T) \circ \pi = (l \cdot J^{\scrD_\tau}_T) 
		\circ \pi = 
		(\eta_F 
		\cdot 
		J_\bbD \cdot J^{\scrD_\tau}_T) \circ (\theta \cdot J^{\scrD_\tau}_T)$. 
		To wit, we have a square
		\begin{center}
			\begin{tikzcd}
				F_\tau t \ar[rr, tail, "\theta_t"] \ar[d, "\pi_t"']& & 
				F_\lambda 
				t \ar[dd, "\eta_F t"]
				\\
				Lt \ar[d, "p_t"'] &
				\\
				F^\#_\tau t \ar[rr, tail, "\varphi_t"'] && F^\#_\lambda t
			\end{tikzcd},
		\end{center}
		which means $\eta_F$ preserves tight parts.
		
		Now, since ${\beta_\tau}_t$ is clearly tight for all $t \in T$, $\beta$ 
		is always a dotted-lax natural transformation. Therefore, the 
		composition $\beta \circ \eta_F$ is also a dotted-lax natural 
		transformation.
	\end{proof}
	
	We attain our main theorem as follows.
	
	\begin{pro}
		\label{pro:precomp_by_unit_is_iso}
		There is an isomorphism of categories
		\begin{equation*}
			[\bbD, \bbF]_{l, \Sigma, T}(F, G) \overset{{\eta_F}^*}{\cong} 
			[\bbD, \bbF](F^\#, G)
		\end{equation*}
		given by the pre-composition by $\eta_F$, natural in $F$ and $G$.
		
		In other words, there is a left adjoint $()^\# \colon [\bbD, \bbF]_{l, 
			\Sigma, T} 
		\to [\bbD, \bbF]$ to the inclusion. 
	\end{pro}
	
	\begin{proof}
		Our goal is to prove that given a dotted-lax natural transformation 
		$\alpha 
		\colon F \to G$, there exists a $2$-natural transformation $\varsigma 
		\colon F^\#_\tau \to G_\tau$ such that the transpose $\beta_\lambda 
		\colon F^\#_\lambda \to G_\lambda$ of $\alpha_\lambda$ together with 
		$\varsigma$ become an 
		$\F$-natural transformation $\beta \colon F^\# \to G$, and that $\beta 
		\circ \eta_F = \alpha$; moreover, for any 
		two $\F$-natural transformation $\beta, \beta' \colon F^\# \to G$, if 
		$\beta \circ \eta_F = \beta' \circ \eta_F$, then $\beta = \beta'$.
		
		By the universal property of the left Kan extension $L$ in 
		\longref{Construction}{constr:tight_part_left_adj}, there exists a 
		unique $2$-natural transformation $g \colon L \to G_\tau$ such that
		\begin{equation*}
			\begin{tikzcd}[row sep = large, column sep = large]
				T \ar[r, hook, "J^{\scrD_\tau}_T"] \ar[dr, hook, 
				"J^{\scrD_\tau}_T"'{name=L}]  & \scrD_\tau \ar[r, "F_\tau"] 
				\ar[Rightarrow, to=2-2, shorten=3mm, 
				"{\alpha_\delta}"'] & \Cat
				\\
				& \scrD_\tau \ar[ur, "{G_\tau}"'{name=M}] 
				& 
			\end{tikzcd}
			=
			\begin{tikzcd}[row sep = large, column sep = large]
				T \ar[r, hook, "J^{\scrD_\tau}_T"] \ar[dr, hook, 
				"J^{\scrD_\tau}_T"'{name=L}]  & \scrD_\tau \ar[r, "F_\tau"] 
				\ar[Rightarrow, to=2-2, shorten=3mm, "{\pi}"'] & 
				\Cat
				\\
				& \scrD_\tau \ar[ur, bend left=15, "{L}"{name=M}] 
				\ar[ur, bend right=30, "{G_\tau}"'{name=R}] 
				\ar[Rightarrow, from=M, to=R, shorten=2.5mm,  dashed,
				"{\exists!g}"] & 
			\end{tikzcd},
		\end{equation*}
		i.e., $(g \cdot J^{\scrD_\tau}_T) \circ \pi = \alpha_\delta$. From the 
		$2$-categorical adjunction, we have $\beta_\lambda 
		\circ \eta_F = \alpha_\lambda$, so by \eqtref{}{eqt:alpha_delta}, the 
		equation then becomes
		\begin{equation*}
			\begin{tikzcd}[row sep = huge, column sep = 4em]
				T \ar[r, hook, "J^{\scrD_\tau}_T"] \ar[dr, hook, 
				"J^{\scrD_\tau}_T"'{name=L}]  & \scrD_\tau \ar[r, "F_\tau"] 
				\ar[Rightarrow, to=2-2, shorten=4.5mm, "{\theta 
					\cdot 
					J^{\scrD_\tau}_T }"'] & \Cat
				\\
				& \scrD_\tau \ar[ur, bend left=15, "\scriptstyle{F_\lambda 
					J_\bbD}"{name=M}] 
				\ar[ur, bend right=20, outer sep = -5pt, 
				"\scriptstyle{F^\#_\lambda 
					J_\bbD}"'{name=R}] 
				\ar[Rightarrow, from=M, to=R, shorten=3mm, pos=0.3,
				"\scriptscriptstyle{\eta_F 
					\cdot 
					J_\bbD}"] \ar[ur, bend right=90, pos=0.6, 
				"\scriptstyle{G_\lambda 
					J_\bbD}"'{name=Z}] \ar[Rightarrow, from=R, to=Z, 
				shorten=1.5mm,  
				"\scriptscriptstyle{\beta_\lambda \cdot J_\bbD}"'] & 
			\end{tikzcd}
			=
			\begin{tikzcd}[row sep = huge, column sep = 4em]
				T \ar[r, hook, "J^{\scrD_\tau}_T"] \ar[dr, hook, 
				"J^{\scrD_\tau}_T"'{name=L}]  & \scrD_\tau \ar[r, "F_\tau"] 
				\ar[Rightarrow, to=2-2, shorten=4.5mm, "{\pi}"'] & 
				\Cat
				\\
				& \scrD_\tau \ar[ur, bend left=15, "{L}"{name=M}] 
				\ar[ur, bend right=20, outer sep=-2pt,
				"\scriptstyle{G_\tau}"'{name=R}] 
				\ar[Rightarrow, from=M, to=R, shorten=2.25mm,  
				"{g}"]  \ar[ur, bend right=90, outer sep=-2pt,
				"\scriptstyle{G_\lambda 
					J_\bbD}"'{name=Z}] \ar[Rightarrow, from=R, to=Z, 
				shorten=2mm,  start anchor={[yshift=1mm]}, end 
				anchor={[yshift=1mm]},
				"{\gamma}"']  & 
			\end{tikzcd}.
		\end{equation*}
		
		Besides, from the construction of $l$ in 		
		\longref{Construction}{constr:tight_part_left_adj}, we have
		\begin{equation*}
			\begin{tikzcd}[row sep = large, column sep = large]
				T \ar[r, hook, "J^{\scrD_\tau}_T"] \ar[dr, hook, 
				"J^{\scrD_\tau}_T"'{name=L}]  & \scrD_\tau \ar[r, "F_\tau"] 
				\ar[Rightarrow, to=2-2, shorten=3mm, "\scriptstyle{\theta 
					\cdot 
					J^{\scrD_\tau}_T }"'] & \Cat
				\\
				& \scrD_\tau \ar[ur, bend left=15, "\scriptstyle{F_\lambda 
					J_\bbD}"{name=M}] 
				\ar[ur, bend right=30, "\scriptstyle{F^\#_\lambda 
					J_\bbD}"'{name=R}] 
				\ar[Rightarrow, from=M, to=R, shorten=3.25mm, pos=0.3,
				"\scriptscriptstyle{\eta_F 
					\cdot 
					J_\bbD}"] & 
			\end{tikzcd}
			=
			\begin{tikzcd}[row sep = large, column sep = large]
				T \ar[r, hook, "J^{\scrD_\tau}_T"] \ar[dr, hook, 
				"J^{\scrD_\tau}_T"'{name=L}]  & \scrD_\tau \ar[r, "F_\tau"] 
				\ar[Rightarrow, to=2-2, shorten=3mm, "{\pi}"'] & 
				\Cat
				\\
				& \scrD_\tau \ar[ur, bend left=15, "{L}"{name=M}] 
				\ar[ur, bend right=30, "{F^\#_\lambda 
					J_\bbD}"'{name=R}] 
				\ar[Rightarrow, from=M, to=R, shorten=3mm, 
				"{l}"] & 
			\end{tikzcd}.
		\end{equation*}
		
		Combining the two equations, we obtain
		\begin{equation*}
			\begin{tikzcd}[row sep = huge, column sep = 4em]
				T \ar[r, hook, "J^{\scrD_\tau}_T"] \ar[dr, hook, 
				"J^{\scrD_\tau}_T"'{name=L}]  & \scrD_\tau \ar[r, "F_\tau"] 
				\ar[Rightarrow, to=2-2, shorten=4.5mm, "{\pi}"'] & 
				\Cat
				\\
				& \scrD_\tau \ar[ur, bend left=15, "{L}"{name=M}] 
				\ar[ur, bend right=20, outer sep=-2pt,
				"\scriptstyle{G_\tau}"'{name=R}] 
				\ar[Rightarrow, from=M, to=R, shorten=2.25mm,  
				"{g}"]  \ar[ur, bend right=90, 
				"\scriptstyle{G_\lambda 
					J_\bbD}"'{name=Z}] \ar[Rightarrow, from=R, to=Z, 
				shorten=2mm,  start anchor={[yshift=1mm]}, end 
				anchor={[yshift=1mm]},
				"{\gamma}"']  & 
			\end{tikzcd}
			=
			\begin{tikzcd}[row sep = huge, column sep = 4em]
				T \ar[r, hook, "J^{\scrD_\tau}_T"] \ar[dr, hook, 
				"J^{\scrD_\tau}_T"'{name=L}]  & \scrD_\tau \ar[r, "F_\tau"] 
				\ar[Rightarrow, to=2-2, shorten=4.5mm, "{\pi}"'] & 
				\Cat
				\\
				& \scrD_\tau \ar[ur, bend left=15, "{L}"{name=M}] 
				\ar[ur, bend right=20, outer sep=-5pt,
				"\scriptstyle{F^\#_\lambda J_\bbD}"'{name=R}] 
				\ar[Rightarrow, from=M, to=R, shorten=2.7mm,  
				"{l}"]  \ar[ur, bend right=90, pos=0.6, 
				"\scriptstyle{G_\lambda 
					J_\bbD}"'{name=Z}] \ar[Rightarrow, from=R, to=Z, 
				shorten=1.25mm,  
				"\scriptscriptstyle{\beta_\lambda \cdot J_\bbD}"']  & 
			\end{tikzcd}.
		\end{equation*}
		
		Now the universal property of $L$ forces $\gamma \circ g = 
		(\beta_\lambda \cdot J_\bbD) 
		\circ l$. This means that for any object $d$ in 
		$\bbD$, we have a commutative square
		\begin{center}
			\begin{tikzcd}[column sep=scriptsize]
				Ld \ar[rr, "g_d"] \ar[dd, two heads, "p_d"'] & & G_\tau d 
				\ar[dd, tail, "\gamma_d"]
				\\
				& &
				\\
				F^\#_\tau d \ar[r, tail, "\varphi_d"'] \ar[to=1-3, dashed, 
				"\exists!\varsigma_d"'] & F^\#_\lambda d \ar[r, 
				"\beta_\lambda d"'] & G_\lambda d
			\end{tikzcd},
		\end{center}
		and since $p_d$ is surjective-on-objects and $\gamma_d$ is a full 
		embedding, there exists a unique 
		lift $\varsigma_d \colon F^\#_\tau d \to G_\tau d$.
		
		Let $f, f_i \colon c \to d$ be tight morphisms in $\scrD_\tau$ for $i = 
		1, 2$, and $m \colon f_1 \Rightarrow f_2$ be a $2$-cell in $\bbD$.  
		In the following diagram:
		\begin{center}
			\begin{tikzcd}[row sep=tiny, column sep=normal]
				F^\#_\tau c \ar[dr, tail, "\varphi_c"] \ar[ddd, bend right=20 
				,"F^\#_\tau 
				f_1"'{name=A}] \ar[ddd, bend left=20 
				,"F^\#_\tau 
				f_2"{name=B}] \ar[Rightarrow, from=A, to=B, shorten=1.25mm,
				"\scriptscriptstyle{F^\#_\tau m}"] 
				\ar[rrr, dashed, "\varsigma_c"]
				& & & G_\tau c \ar[ddl, tail, "\gamma_c"'] \ar[ddd, bend 
				left=20, "G_\tau 
				f_2"{name=H}] \ar[ddd, bend right=20, "G_\tau f_1"'{name=G}] 
				\ar[Rightarrow, 
				from=G, to=H, shorten=1.25mm,
				"\scriptscriptstyle{G_\tau m}"]
				\\
				& F^\#_\lambda c \ar[dr, "\beta_\lambda c"] &  &
				\\
				& & G_\lambda c  &
				\\
				F^\#_\tau d \ar[dr, tail, "\varphi_d"'] \ar[rrr, dashed, 
				"\varsigma_d"'] & & & G_\tau d \ar[ddl, 
				tail, "\gamma_d"]
				\\
				& F^\#_\lambda d \ar[dr, "\beta_\lambda d"'] \ar[from=2-2, bend 
				left=20,
				crossing 
				over, "F^\#_\lambda f_2"{name=D}]  \ar[from=2-2, bend right=20,
				crossing 
				over, "F^\#_\lambda f_1"'{name=C}] \ar[Rightarrow, from=C, 
				to=D, shorten=1.25mm,
				"\scriptscriptstyle{F^\#_\lambda m}"]  &  &
				\\
				& & G_\lambda d \ar[from=3-3, bend left=20, crossing over, 
				pos=0.55,
				"G_\lambda f_2"{name=F}] \ar[from=3-3, bend right=20, crossing 
				over, 
				pos=0.55,
				"G_\lambda f_1"'{name=E}] \ar[Rightarrow, from=E, 
				to=F, shorten=1.25mm,
				"\scriptscriptstyle{G_\lambda m}"] &
			\end{tikzcd},
		\end{center}
		all the parallelograms in the front, i.e., those constituted by solid 
		arrows, commute, because $\varphi$, $\beta_\lambda \cdot J_\bbD$, and 
		$\gamma$ are all $2$-natural transformations, and the $1$-components of 
		$\varphi$ and $\gamma$ are monomorphisms in $\Cat$. Thus, the rectangle 
		at the 
		back, i.e., that constituted by two dashed arrows, also commutes. This 
		amounts to $G_\tau f \varsigma_c = \varsigma_d F^\#_\tau f$ and $G_\tau 
		m \varsigma_c = \varsigma_d F^\#_\tau m$. Consequently, 
		$\{\varsigma_d\}_{d \in \bbD}$ assemble to a $2$-natural transformation 
		$\varsigma \colon F^\#_\tau \to G_\tau$.
		
		From the above, we have $(\beta_\lambda \cdot J_\bbD) \circ \varphi = 
		\gamma \circ \varsigma$, which is exactly 
		\eqtref{}{eqt:tight_mor_between_weights}. Therefore, $\beta_\lambda$ 
		and $\varsigma$ together form an $\F$-natural transformation $\beta 
		\colon F^\# \to G$.
		
		Next, we check that for $t \in T$, $\varsigma_c {\eta_F}_t = 
		{\alpha_\lambda}_t$. From the $2$-categorical adjunction, we already 
		have $\beta_\lambda \circ \eta_F = \alpha_\lambda$.
		
		Indeed, for any $t \in T$, we have the following commutative diagram:
		\begin{center}
			\begin{tikzcd}[row sep=small]
				F_\tau t \ar[rr, tail, "\theta_t"] \ar[d, "p_t \pi_t"'] & & 
				F_\lambda t \ar[d, "{\eta_F}_t"]
				\\
				F^\#_\tau t \ar[d, "\varsigma_t"'] \ar[rr, "\varphi_t"'] & & 
				F^\#_\lambda \ar[d, "{\beta_\lambda}_t"]
				\\
				G_\tau t \ar[rr, tail, "\gamma_t"'] & & G_\lambda t
			\end{tikzcd},
		\end{center}
		where the top diagram is constructed in the proof of 
		\longref{Proposition}{pro:unit_dotted}, and the bottom diagram follows 
		immediately from our above construction of $\varsigma_t$. Together with 
		\eqtref{}{eqt:alpha_delta}, we obtain
		\begin{align*}
			\gamma_t {\alpha_\delta}_t = {\alpha_\lambda}_t \theta_t
			= {\beta_\lambda}_t {\eta_F}_t \theta_t
			= \gamma_t \varsigma_t p_t \pi_t,
		\end{align*}
		now since $\gamma_t$ is a monomorphism, we deduce that
		\begin{equation*}
			{\alpha_\delta}_t = \varsigma_t p_t \pi_t.
		\end{equation*}
		Altogther, we obtain $\beta \circ \eta_F = \alpha$.
		
		Finally, let $\beta, \beta' \colon F^\# \to G$ be two $\F$-natural 
		transformations such that $\beta \circ \eta_F = \beta' \circ \eta_F$.
		
		In particular, $\beta_\lambda \circ \eta_F = \beta'_\lambda \circ 
		\eta_F$, which implies $\beta_\lambda = \beta'_\lambda$. Hence, we have 
		for any $d \in \ob\bbD$,
		\begin{align*}
			\gamma_d {\beta_\tau}_d = {\beta_\lambda}_d \varphi_d
			= {\beta'_\lambda}_d \varphi_d
			= \gamma_d {\beta'_\tau}_d.
		\end{align*}
		Since $\gamma_d$ is monic, this implies ${\beta_\tau}_d = 
		{\beta'_\tau}_d$; as $d$ is arbitrary, we obtain ${\beta_\tau} = 
		{\beta'_\tau}$.	As a consequence, $\beta = \beta'$.
	\end{proof}
	
	\begin{theorem}
		\label{thm:F_left_adj}
		Let $(\bbD, \Sigma, T)$ be a dotted $\F$-category. The dotted-lax limit 
		of an $\F$-functor $S \colon \bbD \to 
		\bbA$ has the same universal property as the $\F$-weighted limit 
		$\{\triangle(\mathbf{1})^\#, 
		S\}$.
	\end{theorem}
	
	\begin{proof}
		We have
		\begin{align*}
			\bbA(A, \dotlim{l}{S}) &\cong [\bbD, \bbA]_{l, \Sigma, 
				T}(\triangle(A), S)
			\\
			&\cong [\bbD, \bbF]_{l, \Sigma, T}(\triangle(\mathbf{1}), \bbA(A, S-))
			\\
			&\cong [\bbD, \bbF]({\triangle(\mathbf{1})}^\#, \bbA(A, S-))
			\\
			&\cong \bbA(A, \{\triangle(\mathbf{1})^\#, S\}). \qedhere
		\end{align*}
	\end{proof}
	
	We proceed to show the converse: any $\F$-weighted limit can be 
	equivalently 
	viewed as a dotted $2$-limit.
	
	Let $\bbA, \bbD$ be $\F$-categories, $\Phi \colon \bbD \to \bbF$ be an 
	$\F$-weight and $S \colon \bbD \to \bbA$ be an $\F$-functor. Consider the 
	$\F$-category of elements of $\Phi$ with
	
	\begin{enumerate}
		\item [$\bullet$] objects the pairs $(D \in \bbD, \delta \in (\Phi 
		D)_\lambda)$;
		\item [$\bullet$] loose morphisms $(D_1, \delta_1) \leadsto (D_2, 
		\delta_2)$ given by the pairs $(d \colon D_1 \leadsto 
		D_2 \in \scrD_\lambda, \omega 
		\colon 
		Wd \delta_1 \to \delta_2)$;
		\item [$\bullet$] tight morphisms $(D_1, \delta_1) \to (D_2, \delta_2)$ 
		given by the pairs $(d \colon D_1 \to D_2 \in 
		\scrD_\tau, \omega 
		\colon 
		Wd \delta_1 \to \delta_2)$;
		\item [$\bullet$] $2$-cells $(d, \omega) \Rightarrow (d', \omega')$ 
		given by the $2$-cells $\alpha \colon d 
		\Rightarrow d'$ such that $\omega' W\alpha \delta_1 = \omega$,
	\end{enumerate}
	where the composition of two morphisms is given by the same rule as in the 
	$2$-categorical case. Similar to the $2$-category of elements, we also have 
	a projection $\F$-functor
	\begin{align*}
		P \colon \bbEl{\Phi} &\to \bbD.
	\end{align*}
	
	\begin{rk}
		In \longcite{Lambert:2020}, Lambert presented a proof of 
		$2$-categories, 
		$2$-functors, lax natural transformations, and modifications forming a 
		lax $3$-category $\twoCat_l$. Using a very similar argument, we see 
		that $\F$-categories, $\F$-functors, lax natural transformations 
		between loose parts, and modifications form a lax $3$-category 
		$\FCat_{l, \lambda}$.
		
		Following the definition of lax comma objects in 
		\longcite{Mesiti:2023}, one can easily check that $\bbEl{\Phi}$ is the 
		lax comma object 
		of 
		$\triangle(\mathbf{1}) \colon 1 \to \bbF$ and 
		$\Phi \colon \bbD \to \bbF$
		\begin{center}
			\begin{tikzcd}
				\bbEl{\Phi} \ar[r, "P"] \ar[d, "!"'] & \bbD \ar[d, "\Phi"]
				\\
				1 \ar[Rightarrow, to=1-2, 
				shorten=5mm, "\mu"']
				\ar[r, "\triangle(\mathbf{1})"'] & \bbF
			\end{tikzcd}
		\end{center}
		in $\FCat_{l, \lambda}$.
	\end{rk}
	
	\begin{pro}
		\label{pro:weighted_to_dotted}
		Let $\Sigma := \{(d \in \mor \bbD, 1)\}$ be the class of morphisms, 
		and $T := \{(D, \delta \in (\Phi D)_\tau)\}$ be the class of objects in 
		$\bbEl{\Phi}$. There is an isomorphism
		\begin{equation*}
			[\bbD, \bbF](\Phi, \bbA(A, S-)) \cong [\bbEl{\Phi}, \bbA]_{l, 
				\Sigma, T}(\triangle(A), SP)
		\end{equation*}
		in $\bbF$, or equivalently, there is an isomorphism of categories
		\begin{equation*}
			[\bbD, \bbF]_\lambda(\Phi, \bbA(A, S-)) \cong [\bbEl{\Phi}, 
			\bbA]_{l, 
				\Sigma, T, \lambda}(\triangle(A), SP),
		\end{equation*}
		which restricts to the tight parts
		\begin{equation*}
			[\bbD, \bbF]_\tau(\Phi, \bbA(A, S-)) \cong [\bbEl{\Phi}, \bbA]_{l, 
				\Sigma, T, \tau}(\triangle(A), SP).
		\end{equation*}
	\end{pro}
	
	\begin{proof}
		The isomorphism for the loose parts are given by 
		\longref{Proposition}{pro:weighted_to_sigma}. It remains to check that 
		the tight morphisms on both sides correspond to each other.
		
		Recall that in the proof of 
		\longref{Proposition}{pro:weighted_to_sigma}, we defined a pair of 
		inverse functors
		\begin{align*}
			\label{functor:G_dotted}
			G \colon [\scrD_\lambda, \Cat](\Phi_\lambda, \scrA_\lambda(A, 
			S_\lambda-)) &\to [\scrEl{\Phi}, \Cat]_{l, 
				\Sigma}(\triangle(\mathbf{1}), \scrA_\lambda(A, S_\lambda P-)),
			\\
			\beta &\mapsto \beta P \circ \mu, \nonumber
			\\
			\Gamma \colon \beta_1 \to \beta_2 &\mapsto \Gamma P \circ \mu, 
			\nonumber
		\end{align*}
		and
		\begin{align*}
			G' \colon  [\scrEl{\Phi}, \Cat]_{l, 
				\Sigma}(\triangle(\mathbf{1}), \scrA_\lambda(A, S_\lambda P-)) &\to 
			[\scrD_\lambda, \Cat](\Phi_\lambda, \scrA_\lambda(A, 
			S_\lambda-)),
			\\
			\alpha &\mapsto \beta = \{\alpha_{(\triangle(D), -)}\}_{D \in 
				\twoD},
			\\
			\Lambda \colon \alpha_1 \to \alpha_2 &\mapsto \Theta = 
			\{\{\Lambda_{(D, 
				\delta)} \}_{\delta \in WD}\}_{D \in \twoD}.
		\end{align*}
		
		Now suppose $\beta \colon \Phi_\lambda \to \scrA_\lambda(A, S_\lambda 
		-)$ is tight, i.e., $\beta$ is a $2$-natural transformation such that 
		for any $\delta \in (\Phi D)_\tau$, $\beta_D (\delta) \colon A \to 
		S_\lambda D$ 
		is tight in $\bbA$.	Clearly, $G\beta = (\beta \circ P) \circ \mu =: 
		\alpha$ 
		is a marked-lax nat 
		transformation $\triangle(\mathbf{1}) \xrightarrow{\alpha} \scrA_\lambda(A, 
		S_\lambda P -)$. Note that $\alpha_{(D, \delta)}$ is given by the 
		composite $\beta_D \circ \mu_{(D, \delta)}$. If $(D, \delta) \in T$, 
		then that means $\mu_{(D, \delta)}$ always picks out an object in 
		$(\Phi 
		D)_\tau$, therefore $\beta_D (\delta)$ is tight by the assumption, 
		hence is $\alpha_{(D, 
			\delta)}$.
		
		Conversely, let $\alpha \colon \triangle(\mathbf{1}) \to \scrA_\lambda(A, 
		S_\lambda P -)$ be a dotted natural transformation, and so $G'(\alpha) 
		=: 
		\beta_\alpha$ is a $2$-natural transformation. Now for any $\delta \in 
		(\Phi D)_\tau$, we have $(D, \delta) \in T$, thus $\beta_{\alpha_D}(d) 
		= \alpha_{(D, \delta)}$ is tight.
	\end{proof}
	
	\begin{theorem}
		Let $\bbA, \bbD$ be $\F$-categories. Let $\Phi \colon \bbD \to \bbF$ be 
		an $\F$-weight and $S \colon \bbD 
		\to 
		\bbA$ be an $\F$-functor. Let $\Sigma := \{(d \in \mor \bbD, 1)\}$ be 
		the class of 
		morphisms, and $T := \{(D, \delta \in (\Phi D)_\tau)\}$ be the 
		collection 
		of objects in the $\F$-category of elements $\bbEl{\Phi}$ of $\Phi$, so 
		that $(\bbEl{\Phi}, \Sigma, T)$ is a dotted $\F$-category. 
		The 
		$\F$-weighted limit $\{\Phi, S\}$ has the same universal property as 
		the 
		dotted-lax limit of the $\F$-functor $SP \colon \bbEl{\Phi} \to 
		\bbA$.
	\end{theorem}
	
	\begin{proof}
		By \longref{Proposition}{pro:weighted_to_dotted}, we have a chain of 
		isomorphisms
		\begin{align*}
			\bbA(A, \{\Phi, S\}) &\cong [\bbD, \bbF](\Phi, \bbD(A, S-))
			\\
			&\cong [\bbEl{\Phi}, \bbA]_{l, \Sigma}(\triangle(A), SP)
			\\
			&\cong \bbA(A, \dotlim{l}{SP}). \qedhere
		\end{align*}
	\end{proof}
	
	\section{Discussions}
	\label{sec:discuss}
	
	In \longcite{Szyld:2019}, Szyld showed the lifting theorem for 
	$2$-categorical limits using marked limits, which provides extra 
	information on the projections than the classical proof. Indeed, Szyld's 
	theorem can be rephrased in a cleaner manner in terms of enhanced 
	$2$-category theory and dotted $2$-limits, for the case of strict and lax (or 
	colax) $T$-morphisms.
%	\newpage
	\begin{defi}
		\label{def:PIE}
		A dotted $\F$-category $(\bbD, \Sigma, \Gamma)$ is said to be 
		\emph{weakly PIE-indexing} precisely when each connected component of  
		the full sub-$\F$-category $\Sigma \subseteq \bbD$ has an initial 
		object, which is contained in $\Gamma$.
		
		A weakly PIE-indexing dotted $\F$-category is said to be \emph{strongly 
			PIE-indexing} precisely if furthermore the unique morphism from 
		the initial object to any object is always tight.
	\end{defi}
	
	\begin{nota}
		We denote by $N$ the collection of these initial objects.
	\end{nota}
	
	Now, {\cite[Theorem 3.3]{Szyld:2019}} can be reformulated as
	\begin{theorem}
		\label{thm:lift_p}
		Let $(\bbD, \Sigma, \Gamma)$ be a weakly PIE-indexing dotted 
		$\F$-category, and let $T$ be a $2$-monad on a $2$-category $\twoA$, 
		viewed as a chordate 
		$\F$-category. Let $S \colon \bbD \to \FTAlg_{s, p}$ 
		be an $\F$-functor. Denote by $U_{s, p} \colon \FTAlg_{s, p} \to \twoA$ 
		the 
		underlying $\F$-functor. If $\dotlim{l}{U_{s, p}S}$ exists in $\twoA$, 
		then 
		$\dotlim{l}{S}$ exists in $\FTAlg_{s, p}$, and is preserved by $U_{s, 
			p}$.
		
		Moreover, the projections $\{\pi_n\}_{n \in N}$ of $\dotlim{l}{S}$ are 
		tight and jointly 
		detect tightness.
	\end{theorem}
	
	\begin{theorem}
		\label{thm:lift_l,c}
		Let $(\bbD, \Sigma, \Gamma)$ be a strongly PIE-indexing dotted 
		$\F$-category, and let $T$ be a $2$-monad on a $2$-category $\twoA$, 
		viewed as a chordate 
		$\F$-category. Let $S \colon \bbD \to \FTAlg_{s, c}$ 
		be an $\F$-functor. Denote by $U_{s, c} \colon \FTAlg_{s, c} \to \twoA$ 
		the 
		underlying $\F$-functor. If $\dotlim{l}{U_{s, c}S}$ exists in $\twoA$, 
		then 
		$\dotlim{l}{S}$ exists in $\FTAlg_{s, c}$, and is preserved by $U_{s, 
			c}$.
		
		Moreover, the projections $\{\pi_n\}_{n \in N}$ of $\dotlim{l}{S}$ are 
		tight and jointly 
		detect tightness.
		
		Dually, If $\dotlim{c}{U_{s, l}S}$ exists in $\twoA$, then 
		$\dotlim{c}{S}$ exists in $\FTAlg_{s, l}$, and is preserved by $U_{s, 
			l}$, and the projections of 
		$\dotlim{c}{S}$ are 
		tight and jointly 
		detect tightness.
	\end{theorem}
	
	It is shown in \longcite{LS:2012} that $w$-rigged $\F$-weighted limits 
	such as $w$-rigged inserters, $w$-rigged equifiers, $w$-rigged 
	$w$-descent objects and $w$-rigged comma objects all lift to 
	$\FTAlg_{s, w'}$, where $w' = c$ if $w = l$ and vice versa. Indeed, 
	they are precisely the limits that lift. Looking at 
	the examples of $p$-rigged limits illustrated in 
	\longref{Section}{subsec:eg}, it 
	is clear that the assumptions in \longref{Theorem}{thm:lift_p} are met;
	looking at 
	the examples of $l$-rigged and $c$-rigged limits illustrated in 
	\longref{Section}{subsec:eg}, it is 
	immediate that 
	the assumptions in \longref{Theorem}{thm:lift_l,c} are fulfilled. In 
	other 
	words, through the lens of dotted $2$-limits, we see that Szyld's lifting 
	theorem also applies to some of the $\F$-categorical limits.
	
	Nevertheless, the alternating projective limits in 
	\longref{Example}{eg:alt}, which is shown to 
	lift in \longcite{LS:2012}, do not satisfy the assumptions: there is no 
	initial object in the indexing dotted $\F$-category, as $m$ 
	tends to infinity.
	
	A possible direction and further exploration on dotted $2$-limits is to 
	characterise $w$-rigged 
	$\F$-weighted limits in the language of dotted $2$-limits, so that a more 
	elementary and explicit description of rigged limits can be found.

	\bibliographystyle{alpha}

	%	\bibliography{../joannaref}
	%	\printbibliography
\end{document}